\documentclass[a4paper,11pt]{article}
\usepackage[utf8]{inputenc}
\usepackage{latexsym}
\usepackage{amsfonts}
\usepackage{amssymb}

\usepackage{enumitem} 
\usepackage{braket}
\usepackage{hyperref}

\usepackage[T1]{fontenc}
\usepackage{amsmath}
\usepackage{amssymb}
\usepackage{amsfonts}
\usepackage{amsthm}
\usepackage{indentfirst}
\usepackage{psfrag}
\usepackage{amscd,stmaryrd,latexsym}
\usepackage[margin=1.3in]{geometry}
\usepackage[pdftex]{graphicx}
\newtheorem{thm}{Theorem} 
\newtheorem{cor}{Corollary}

\newtheorem{Prop}{Proposition}[section]
\newtheorem*{St*}{Statement}

\newtheorem{Dfi}{Definition}

\newtheorem*{theorem*}{Theorem}
\newtheorem*{corollary*}{Corollary}
\theoremstyle{remark}

\theoremstyle{definition}

\theoremstyle{remark}
\newtheorem{oss}{Remark}[section]

\newcommand{\be}{\begin{equation}}
\newcommand{\ee}{\end{equation}}
\newcommand{\R}{\mathbb{R}}

\newcommand{\spt}[1]{\text{spt}\,\|#1\|}

\newcommand\res{\mathop{\hbox{\vrule height 7pt width .5pt depth 0pt
\vrule height .5pt width 6pt depth 0pt}}\nolimits}

\def\eps{\mathop{\varepsilon}}

\def\Hc{\mathop{\mathcal{H}}}

\def\Om{\Omega}

\def\p{\partial}

\def\eps{\mathop{\varepsilon}}

\def\Om{\Omega}

\def\p{\partial}

\DeclareMathAlphabet{\mathscr}{OT1}{pzc}{m}{it}

\begin{document} 

\title{\textbf{Embeddedness of liquid-vapour interfaces in stable equilibrium}}
\author{Costante Bellettini\\
University College London}
\date{}

\maketitle

\begin{abstract}
We consider a classical (capillary) model for a one-phase liquid in equilibrium. The liquid (e.g.~water) is subject to a volume constraint, it does not mix with the surrounding  vapour (e.g.~air), it may come into contact with solid supports (e.g.~a container), and is subject to the action of an analytic potential field (e.g.~gravity). The region occupied by the liquid is described as a set of locally finite perimeter (Caccioppoli set) in $\R^3$; no a priori regularity assumption is made on its boundary. The (twofold) scope in this note is to propose a weakest possible set of mathematical assumptions that sensibly describe a condition of stable equilibrium for the liquid-vapour interface (the capillary surface), and to infer from those that this interface is a smoothly embedded analytic surface. 
(The liquid-solid-vapour junction, or free boundary, can be present but is not analysed here.) The result relies fundamentally on the recent varifold regularity theory developed in \cite{BW}, \cite{BW1}, and on the identification of a suitable formulation of the stability condition.
\end{abstract}

\section{Introduction}
\label{intro}
Given a constant-density, incompressible liquid sitting in a container and subject to a potential field, we denote by $\Om \subset \R^{3}$ the open set identified with the interior of the container and by $E \subset \Om$ the region occupied by the liquid, for the moment assuming $\p E \cap \Om$ smooth and embedded up to $\p \Om$, and $\p \Om$ sufficiently regular, e.g.~$C^1$. The liquid and the surrounding vapour cannot mix. Following a classical capillary model for one-phase liquids (see e.g.~\cite{Finn}, \cite{Maggi}, \cite{Mellet}), the equilibrium condition is obtained by imposing a volume-constrained stationarity condition with respect to the free energy

$$\mathcal{F}(E) = {\Hc}^2(\p E \cap \Om) + \sigma {\Hc}^{2}(\p E \cap \p \Om) + \int_{E} g \,d\mathcal{L}^{3},$$
where $\sigma \in (-1, 1)$ is the adhesive coefficient between the liquid and the walls of the container, ${\Hc}^2$ is the Hausdorff $2$-dimensional measure\footnote{The measure $\mathcal{H}^2$ agrees with the usual surface measure when restricted to a smooth surface.} and $g$ is the potential (for example, $g$ can be the gravitational potential in $\R^3$ given by $g=g_0 \rho \,x_3$, where $\rho$ is the constant density of the liquid and $g_0$ is gravity on Earth). The first term encodes the effect of surface tension on the liquid-vapour interface $\p E \cap \Om$, usually referred to as capillary surface.

The formulation in terms of the free energy $\mathcal{F}$ goes back to the work of Gauss, where the equilibrium condition for $E$ follows from the virtual work principle, as follows. We imagine a continuous deformation of the region $E$ (made so that the deformed region keeps the same volume and stays inside the container), we compute the free energy of the deformed regions and we impose that $\mathcal{F}$ does not decrease to first order along the deformation. This stationarity condition recovers the Young--Laplace's law, i.e.~the following two conditions: (i) the mean curvature of $\p E \cap \Om$ (liquid-vapour interface) is $(g-\lambda) \nu_E$ for some $\lambda \in \R$ (possibly depending on the connected component of the liquid), where $\nu_E$ is the unit outer normal on $\p E$; (ii) $\nu_E \cdot \nu_{\Om}=\sigma$ at the boundary of $\overline{\p E \cap \Om}$ (liquid-vapour-solid junction), where $\nu_{\Om}$ is the unit outer normal on $\p \Om$. In order to reflect physically observable equilibria, one often additionally imposes a second order condition, namely the fact that the equilibrium configuration is stable, by requiring that $\mathcal{F}$ does not increase to second order along the deformation, that is, imposing the non-negativity of the second variation of $\mathcal{F}$. 

\medskip

A more modern formulation of the above problem models the region $E$ occupied by the liquid as a Caccioppoli set\footnote{See Appendix \ref{Caccioppoli} for some reminders.} (rather than a set with smooth boundary), see \cite{Finn}, \cite{Maggi}, \cite{Mellet}. This is partly motivated by the need to pose the variational problem in a class where $E$ has just enough structure to ensure that the energy $\mathcal{F}$ is well-defined, but no further regularity is assumed. Concretely, this will include in the model all equilibrium configurations, and a priori there may be equilibria for which the liquid-vapour interface lacks smoothness. From a mathematical point of view, in order to answer existence questions it is necessary to work in classes that enjoy good compactness properties at the price of possessing less regularity (for example, to use the direct method of calculus of variation or to use min-max arguments): good compactness properties typically fail for smooth submanifolds, and one is required to work with weaker notions that enjoy them (e.g.~reduced boundaries of Caccioppoli sets, currents, varifolds). In the case of Caccioppoli sets that we are interested in, the regularity (smoothness) properties of $\p E$, more precisely of the reduced boundary $\p^* E$, under suitable variational hypotheses, become then a fundamental issue to address.\footnote{The reduced boundary $\p^* C$ is the measure-theoretic notion of boundary for a Caccioppoli set $C$. It is in general merely a rectifiable set, and as such can be very irregular. Any regularity information must be deduced from the specific variational assumptions.}

\medskip

In this note we propose a set of mathematical assumptions on the Caccioppoli set $E$ that reflect a stable equilibrium condition, and that can be considered as weak as possible. We will discuss why such assumptions are mathematically necessary and why they sensibly encode, through the virtual work principle, the notion of stable equilibrium. Some care will be required in stating the variational assumptions via the virtual work principle, since no regularity is assumed on $\p^* E$.
We stress that the stability assumption will involve non-negativity of the second variation (rather than the stronger positivity requirement). Once the assumptions are in place, we will infer from them that, when the potential is analytic, the liquid-vapour interface (i.e.~any portion of $\p^* E$ that does not touch the solid support of the container) is in fact a smoothly embedded (analytic) surface.
The aim is thus to justify, under a minimal set of postulates, the description of a one-phase liquid in stable equilibrium by means of a set whose boundary is smoothly embedded (in the interior of the container). In most of the literature, the liquid-vapour interface is already assumed to be smoothly embedded, in accordance with experimental evidence.

\medskip

We will not be concerned with the free boundary, that is the liquid-vapour-solid junction given by the boundary of the surface-with-boundary $\overline{\p E \cap \Om}$; we address only the interior regularity, i.e.~the regularity of the liquid-vapour interface away from the container. It is then clear that the same results apply to a pendant drop hanging from a syringe, a sessile drop sitting on a desk, or, more generally, to any incompressible liquid with constant density, as long as we stay away from any solid support that may be present. Mathematically, this means that we will only be interested in the energy given by the first and third terms of $\mathcal{F}$.

Our set up is then the following. Let $C$ be a Caccioppoli set in $\R^3$ and let $\Om\subset \R^3$ be an open set (concretely, we can think of $\Om$ as any open subset of $\R^3\setminus S$, where $S$ is the closed set representing all solid supports; note that $\overline{C}$ is not necessarily contained in $\Om$) and let $g:\Om\to \R$ be smooth. The free energy of $C$ in $\Om$ is (here $\chi_C$ stands for the characteristic function of $C$, whose distributional gradient is a Radon measure by definition of Caccioppoli sets, see Appendix \ref{Caccioppoli})
$$\mathcal{E}_{\Om}(C) = \text{Per}_{\Om}(C) + \int_{C\cap {\Om}} g \,d\mathcal{L}^{3} = \|D \chi_C\|(\Om) + \int_{C\cap \Om} g \,d\mathcal{L}^{3} .$$
Our aim is to prove the following result (whose precise statement is given in Section \ref{non_ambient}, Theorem \ref{thm:main}):

\begin{theorem*}
Let $\Om \subset \R^3$ be open, $g:\Om\to \R$ analytic, $E$ a Caccioppoli set in $\R^3$. Assume that $E$ is stationary and stable with respect to $\mathcal{E}_{\Om}$ for volume-preserving deformations inside $\Om$. Then $\p^* E \cap \Om = \p E\cap \Om$ is a smoothly embedded analytic surface.
\end{theorem*}

In Section \ref{assumptions} we will give a precise meaning to the variational assumptions, i.e.~the notions of stationarity and stability for volume-preserving deformations. For the moment we summarise them, a bit imprecisely, by saying that we will require:

\noindent (i) the vanishing of the first variation (stationarity) and non-negativity of the second variation (stability) with respect to $\mathcal{E}_{\Om}$ for any one-parameter family of diffeomorphisms of $\Om$ whose initial speed is in $C^1_c(\Om;\R^3)$ (``ambient deformations'') and such that the one-parameter family of Caccioppoli sets in $\Om$ obtained by acting with the diffeomorphisms on an arbitrary ``connected component'' of $E$ has constant volume;

\noindent (ii) that if, in a neighbourhood $U\subset \subset \Om$, $\p E$ is given by one of three precisely defined non-embedded structures (that we will denote by \textbf{(a)}, \textbf{(a')}, \textbf{(b)} in Section \ref{non_ambient}, see also Figure \ref{fig:localstructures}), then $E$ is stationary and stable with respect to specific volume-preseving deformations (that we will call ``coalescence'' and ``break up'', see Figures \ref{fig:touching_sing_deform} and \ref{fig:class_sing_deform}).

\medskip

The ambient deformations considered in (i) above are customary in the literature on variational problems such as the one we consider. However, assumption (i) alone does not imply the embeddedness conclusion. We identify in (ii) the missing variational hypotheses that lead to it. In addition to why they are mathematically necessary, we will also discuss why the volume-preserving deformations considered in (i) and (ii) can be deemed admissible, in the sense that they reflect observed deformations of liquids.

The above theorem generalizes the corresponding regularity results available for volume-constrained perimeter-minimisers (\cite{GonzMassTaman}, \cite{GonzMassTaman2}), replacing the minimising (or locally minimising) condition, which involves the $L^1$-topology on Caccioppoli sets, with the weaker stationarity and stability conditions (i) and (ii). Avoiding the use of the $L^1$ topology, (i) and (ii) permit to consider the admissibility of the virtual deformations in relation to concrete liquids (see Remark \ref{oss:minimisers} for further comments on minimisers).

\medskip

As we will see, the proof of the theorem relies heavily on a corollary of \cite{BW}, \cite{BW1}, that we recall in Theorem \ref{1} of Section \ref{non_ambient}. The more general theory developed in \cite{BW}, \cite{BW1} addresses a class of integral varifolds with prescribed mean curvature that satisfy certain stability conditions. Particular emphasis is given, in \cite{BW}, \cite{BW1}, to the fact that the hypotheses made are ``easily checkable'', particularly in view of applications to problems in differential geometry; a key example of such an application is given with the resolution of the existence problem for prescribed-mean-curvature hypersurfaces in \cite{BW2}. In this note, on the other hand, we capitalise on the ``checkability'' feature again (this time in a different sense, namely by drawing a parallel with observed behaviour of liquids) when we need to check whether certain \textit{specific} configurations (\textbf{(a)}, \textbf{(b)}) are stationary. The optimality of the conclusions in \cite{BW}, \cite{BW1} (see Theorem \ref{1}) then permits us to reduce the embeddedness question to the analysis of a \textit{single} specific configuration (\textbf{(a')}), which we will prove not to be stable (Section \ref{finalproof}). It is here that, through somewhat non-standard computations, we will make key use of the stability condition for ``coalescence'' and ``break up'' deformations (mentioned in (ii) above).

\begin{oss}
While the $3$-dimensional setting is natural for the concrete problem, it is mathematically interesting to question the validity of an analogue of Theorem \ref{thm:main} in arbitrary dimensions; similarly, weaker assumptions on $g$ may be investigated. This will be discussed briefly in Appendix \ref{B}. 
\end{oss}

We point out an application to the case of droplets that sit in a potential field in the absence of any solid support. This problem was a motivation for \cite{CirMaggi}, where the $C^2$ embeddedness of the boundary of the droplets is an assumption. Combining the above theorem with \cite{CirMaggi} and with the curvature estimates of \cite{BCW}, we obtain (the precise statement is given in Corollary \ref{cor:droplets}, Section \ref{droplets}):

\begin{corollary*}
Let $g\in \R^3 \to \R$ be analytic and assume that the Caccioppoli set $E \subset \R^3$ is bounded, connected, and it is stationary and stable with respect to the energy $\mathcal{E}_{\R^3}$ for volume-preserving deformations. If $|E|$ is sufficiently small, then $E$ is a perturbation of a single sphere, that is, $\p E = \p^* E$ is given by the graph of a smooth small function on a sphere.
\end{corollary*}

In other words, droplets tend to remain almost spherical; a configuration where they come into contact and connect to each other by a tiny meniscus is not stable. 

Using the results in \cite{BCW}, in Section \ref{droplets} we also comment more generally (not only in the case of droplets) on the absence of regions of very high curvature in the conclusion of the main theorem.

\subsection{Preliminary heuristic discussion}
\label{heuristic}

The guiding idea in the choice of stationarity and stability assumptions in the main theorem goes back to the implementation of the virtual work principle. For that, one has to decide which ``virtual deformations'' of the liquid are \textit{admissible}. Indeed, the virtual work principle requires a comparison argument, through a deformation of a given bulk of liquid, and it seems natural to admit only those deformations of $E$ that can be thought of as concrete movements of a liquid bulk. We therefore wish to call a deformation admissible if it reflects a deformation that can be concretely induced for a liquid. One should keep in mind that $\p^* E$ is of unknown regularity (for an arbitrary Caccioppoli set $E$) and therefore this task may look ill-posed: the configurations that are concretely realised for liquids are all quite regular (at least at the macroscopic level that the model under consideration is concerned with). We begin with a discussion from which we will extrapolate some of our postulates. 
\medskip

\textit{On connected components.} We start with the case in which $\p E$ is smoothly embedded. Already here, one does not allow all volume-preserving deformations of $E$. For example, if $E$ has two connected components $E_1$ and $E_2$, that sit a positive distance apart, then we do not allow transfer of liquid from one component to the other, rather we require the volume-preserving virtual deformations to separately preserve the volumes of $E_1$ and $E_2$. For instance, $E_1$ and $E_2$ can be two spherical drops, in the case $g\equiv 0$, that do not touch: this is a stable equilibrium. (If we allowed arbitrary volume-preserving deformations of $E$, we would allow one drop to gain volume and the other to lose it. Indeed, if we allowed that, if the initial balls have different radii the stationarity condition would force one sphere to gain volume until the other one disappears; in the case in which the two initial balls have equal radii the same effect would be induced by the stability condition.) These considerations motivate the necessity of Definition \ref{Dfi:vol-pres_Leb} below, which essentially says that liquid can move only within a ``connected component of $E$'' when we perform a virtual deformation.

\medskip

\textit{On coalescence.} We stress that it is however observed experimentally that distinct smooth drops can interact when they are close enough to each other, and the distance at which these forces are felt depends on the characteristics of the liquid and surrounding vapour. An observable effect of these interactions is the fact that, if two drops are brought close to each other in a quasi-static way, once they are close enough a thin connecting neck (meniscus) forms\footnote{A standard miscroscopic explanation of this effect is the following: liquid molecules are attracted to each other, and this results primarily in surface tension, since for molecules at the boundary of a drop the attraction is mainly on one side. If two drops are sufficiently close, then the molecules at the boundary also feel the attraction towards the other drop: this causes the formation of a meniscus connecting the drops. The specific properties of the liquid and of the vapour sorrounding it affect the distance at which these interactions are meaningful.}. Our energy $\mathcal{E}_{\Om}$ does not account for these distant interactions (while other models do, e.g.~Allen-Cahn); however, we wish to keep track of them in the ``limit case'' when the distance between the two drops becomes $0$. From this perspective, we regard the contact set of two smooth drops $E_1$ and $E_2$ that touch tangentially (along a submanifold of dimension $\leq 1$) as a mathematical idealization of a thin connecting neck, or meniscus.
While the energy $\mathcal{E}_{\Om}$ (just like in the case of positive distance) sees no interaction between $E_1$ and $E_2$ when they touch tangentially, we encode the interaction effects mathematically by \textit{allowing} virtual deformations in which the touching set may be ``smoothed out'' into a thicker neck (we will call such a virtual deformation ``coalescence'', see Fig. \ref{fig:touching_sing_deform}, top row). The virtual work principle for such a deformation will show that surface tension (acting via the perimeter term in $\mathcal{E}_{\Om}$) favours the coalescence of tangentially touching drops (in accordance with experiments). Exploiting stability (particularly the condition that we identify in (ii) that follows the main theorem) will be essential for this, since stationarity alone does not always suffice.

\begin{oss}
We have just discussed ``virtual coalescence'' for distinct \textit{smooth} drops touching, a case in which the deformation can be concretely visualized and thus related to observable behaviour of liquids (and therefore deemed admissible for the virtual work principle). However, we are interested in boundaries of Caccioppoli sets, that are potentially very irregular. What would ``touching drops'' and ``connecting meniscus'' mean in that context, and would it be possible to define ``coalescence'' without knowing what shape we start from? The answer is probably no. As we will see, a key advantage of the ``checkability'' of the hypotheses in \cite{BW}, \cite{BW1} is that we only need to discuss ``virtual coalescence'' in the case of drops with \textit{regular} shapes that touch in a \textit{regular} fashion: in particular, we can concretely visualize the bulks of liquid that come into contact, and it thus makes sense to consider virtual deformations such as coalescence.
\end{oss}

\section{Variational assumptions and main result}
\label{assumptions}
\subsection{Ambient deformations}
\label{ambient}

Restricting, to begin with, to the classical case in which $\p E \cap \Om$ is known to be a smooth embedded surface, the virtual deformations that are used to reach the conclusion that the mean curvature of $\p E \cap \Om$ is given by $(g-\lambda) \nu$ (for some $\lambda\in \R$ possibly depending on the connected component, and with $\nu$ being the unit outer normal) are those induced by ambient diffeomorphisms that fix $E$ outside an arbitrary open set $U \subset \subset \Om$ in which $E\cap U$ is connected. This means that for any one-parameter family $\psi_t$ of ambient diffeomorphisms that keep $\R^{3}\setminus U$ fixed (here $t\in (-\eps, \eps)$, $\psi_t(x)=\psi(t,x)$ is $C^1$ in $t$, $\psi_t:\R^{3} \to \R^{3}$ is a $C^1$ diffeomorphism for every $t$, with $\psi_t = Id$ for $t=0$ and $\psi_t|_{\R^{3}\setminus U} = Id$ for every $t$) and that preserve the volume of $E$ (i.e.~$|\psi_t(E) \cap U|=|E \cap U|$ for all $t$) we require stationarity with respect to $\mathcal{E}_U$ at $t=0$ along the deformation, i.e.
$$\left.\frac{d}{dt}\right|_{t=0}\mathcal{E}_U(\psi_t(E))=0.$$
These deformations reflect the fact that we may slighly perturb the liquid in $U$ ``as a whole'', respecting the volume constraint. Deformations of the type just considered will be called \textit{ambient deformations} and the stationarity condition just stated leads to the mean curvature characterization recalled above (see \cite{Maggi} or Section \ref{A} below).

\medskip

We pass now to the case in which $E$ is only assumed to be a Caccioppoli set. What is an admissible volume-preserving virtual deformation of $E$? The notion of admissible virtual deformation implicitly involves questioning whether such a deformation is meaningful for a concrete liquid. At the same time, as we mentioned earlier, all configurations that are experimentally observed are rather regular, therefore the task to characterise \textit{all} admissible volume-preserving virtual deformations of $E$ is likely ill-posed ($\p E$ and of $\p^* E$ could be a priori extremely irregular). Fortunately we do not need to characterise all meaningful deformations, but only sufficiently many types of deformations to allow the regularity conclusion of the main theorem (Theorem \ref{thm:main} below). 

Deformations that are natural (and indispensable to infer any amount of regularity) are the ambient ones introduced above. The derivative $\left.\frac{\p}{\p t}\right|_{t=0} \psi_t$ is the initial speed by which $E$ moves in $U$; one can think of it as induced by the (short-lasting) action of an external field, that causes a ``slight shaking'' of the liquid bulk. For such a deformation  to make sense we do not need to know anything about the shape of the liquid bulk, because the liquid is perturbed as a whole in the relevant open set. Hence we admit such virtual deformations in our analysis. We only need some care in implementing the volume constraint.
We postulate the following, in accordance with the discussion in Section \ref{heuristic}.

\begin{Dfi}
\label{Dfi:vol-pres_Leb}
Admissible volume-preserving ambient deformations of $E$ in $U$ are those ambient deformations that separately preserve the volume of each connected component of $\overline{E^L} \cap U$, where $E^L$ is the Lebesgue representative of $E$ and $\overline{E^L}$ denotes its closure. 
\end{Dfi}

The choice of the Lebesgue representative (i.e.~the set of points at which the density of $E$ is respect to $\mathcal{L}^{3}$ is $1$) is needed in order to reflect that a virtual deformation allows liquid transfer only through points where there is actually some liquid\footnote{Given two distant Caccioppoli sets $E_1$ and $E_2$ in $\R^{n+1}$, $n\geq 1$, we can always connect them using a curve $\gamma$ and the set $E_1 \cup E_2 \cup \gamma$ is equivalent, as a set of finite perimeter, to $E_1 \cup E_2$. The curve $\gamma$ is of course an artificial addition of liquid: choosing the Lebesgue representative will remove it. Recall that sets of finite perimeter in $\R^{n+1}$ that differ by a $\mathcal{L}^{n+1}$-negligeable set are actually the same set, from the geometric measure theory perspective, so it is enough to work on one representative.}.
To simplify notation \textit{we will assume that $E$ is its own Lebesgue representative}.
The stationarity requirement with respect to ambient volume-preserving deformations, i.e.~the fact that
\begin{equation}
\label{eq:stationarity}
\left.\frac{d}{dt}\right|_{t=0}\mathcal{E}_U(\psi_t(E))=0
\end{equation}
for every admissible ambient volume-preserving deformation $\psi_t$, leads to the condition that (see Section \ref{A}), whenever $U \subset \subset \Om$ is such that $\overline{E}\cap U$ is connected, then $|\p^* E| \res U$ has first variation in $L^\infty$ and generalized mean curvature equal to $g - \lambda$ for some $\lambda \in \R$ (possibly depending on the connected component of $\overline{E} \cap U$). Here $|\p^* E|$ has to be interpreted as the integral varifold obtained by assigning multiplicity $1$ a.e.~on the reduced boundary $\p^* E$, and $|\p^* E|\res U$ denotes its restriction to $U$. Integral varifolds are weak notions of submanifolds and the generalised mean curvature coincides with the usual mean curvature in the case in which the integral varifold ``is actually a submanifold''.
The $L^\infty$ condition on the first variation is not sufficient\footnote{If, for example, we had Lipschitz regularity of $\p E$ we could bootstrap via elliptic theory to $C^{2,\alpha}$ and higher. The example produced in \cite{Brak} shows that the $L^\infty$ condition is too weak for this purpose.}\footnote{In the special case $g\equiv 0$ and assuming that $E$ is bounded and $\Om=\R^{n+1}$, \cite{DelgMag} strikingly proves that, under only this stationarity condition, $E$ must equal a collection of balls with equal radii; the arguments in \cite{DelgMag} exploit however global properties. A \textit{local} regularity result (that goes beyond Allard's theorem) under this stationarity assumption alone does not seem within reach at the moment, and at any rate it would have to allow for non-embedded points, as we will also see below. The local feature is important for us in that it allows the applicability of the result to the interior points of the capillary surface in the presence of solid supports.} to imply enough local regularity for $\p^* E$. Therefore one is led to introduce more restrictive variational assumptions, as in the case treated there, where the equilibrium is assumed to be stable. This also makes sense from a physical point of view, since experimentally observable configurations are stable. The stability condition for volume-preserving ambient deformations $\psi_t$ amounts to the non-negativity of the second variation of $\mathcal{E}_U$ (note that we do not require its strict positivity):
\begin{equation}
\label{eq:stability}
\left.\frac{d^2}{dt^2}\right|_{t=0}\mathcal{E}_U(\psi_t(E))\geq 0.
\end{equation}
The stationarity condition (\ref{eq:stationarity}) for volume-preseving ambient deformations has the following equivalent formulation by means of a Lagrange multiplier (we will prove this in Section \ref{A}): whenever $U \subset \subset \Om$ is such that $\overline{E} \cap U$ is connected, then there exists $\lambda \in \R$ such that $E$ is stationary with respect to the functional $\mathcal{E}_U(E) - \lambda |E \cap U|$ for arbitrary ambient deformations that fix $\R^{3}\setminus U$ (not necessarily volume-preserving ones). In other words, the volume constraint gets encoded in the energy functional. On the contrary, prescribing the non-negativity of the second variation of $\mathcal{E}_U(E) - \lambda |E \cap U|$ for arbitrary ambient deformations is strictly stronger than the stability requirement that we have given in (\ref{eq:stability}) (as explained e.g.~in \cite{BarbDoCarmo}); so we will keep the formulation given in (\ref{eq:stability}).\footnote{Stability for $\mathcal{E}_U(E) - \lambda |E \cap U|$ for arbitrary deformations is usually called strong stability; stability under volume-preserving deformations is usually referred to as weak stability and it implies in particular that the Morse index of $\mathcal{E}_U(E) - \lambda |E \cap U|$ (for arbitrary ambient deformations) is at most $1$.}

\medskip

Having agreed that prescribing the validity of (\ref{eq:stationarity}) and (\ref{eq:stability}) for all volume-preserving ambient deformations is coherent with the description of a stable equilibrium, a natural question is whether these assumtions are sufficient to characterise stable equilibria. However, as we will now describe, certain local structures for $\p E$ are not ruled out by (\ref{eq:stationarity}) and (\ref{eq:stability}), and such configurations are considered with consensus not to describe stable equilibria of the one-phase liquids under consideration. Here are the three local structures:

\begin{description}
 \item[(a)] there exists an open set $U \subset \subset \Om$ such that $\p E \cap U = D_1 \cup D_2$, where $D_1$ and $D_2$ are smooth embedded surfaces in $U$ that intersect tangentially along a smooth $1$-dimensional submanifold. 
\end{description}

\begin{description}
 \item[(a')] there exists an open set $U \subset \subset \Om$ such that $\p E \cap U$ is the union of exactly two smoothly embedded disks $D_1$ and $D_2$ that intersect tangentially at a point. 
\end{description}

\begin{description}
 \item[(b)] there exists an open set $U \subset \subset \Om$ such that $\p E \cap U = \cup_{j=1}^{2N} L_j$, where the $L_j$'s are surfaces-with-boundary that have a common boundary $T$; $L_j$ are smooth away from $T$, $T$ is of class $C^{1,\alpha}$; all the $L_j$'s are $C^{1,\alpha}$ up to their boundary $T$; at least two of the $L_j$'s intersect transversely.
\end{description}

\begin{figure}[h]
\centering
 \includegraphics[width=6cm]{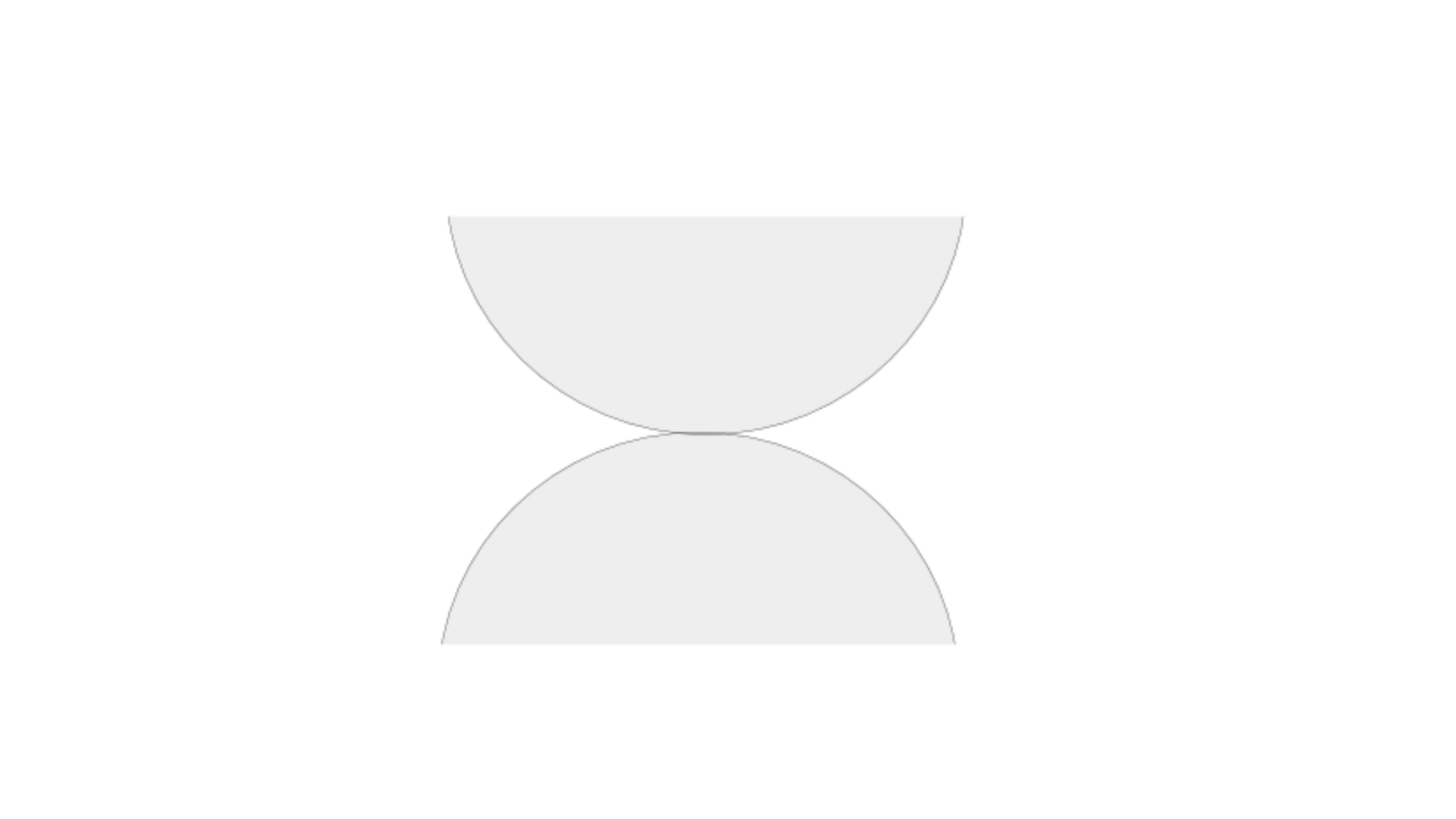} 
 \includegraphics[width=6cm]{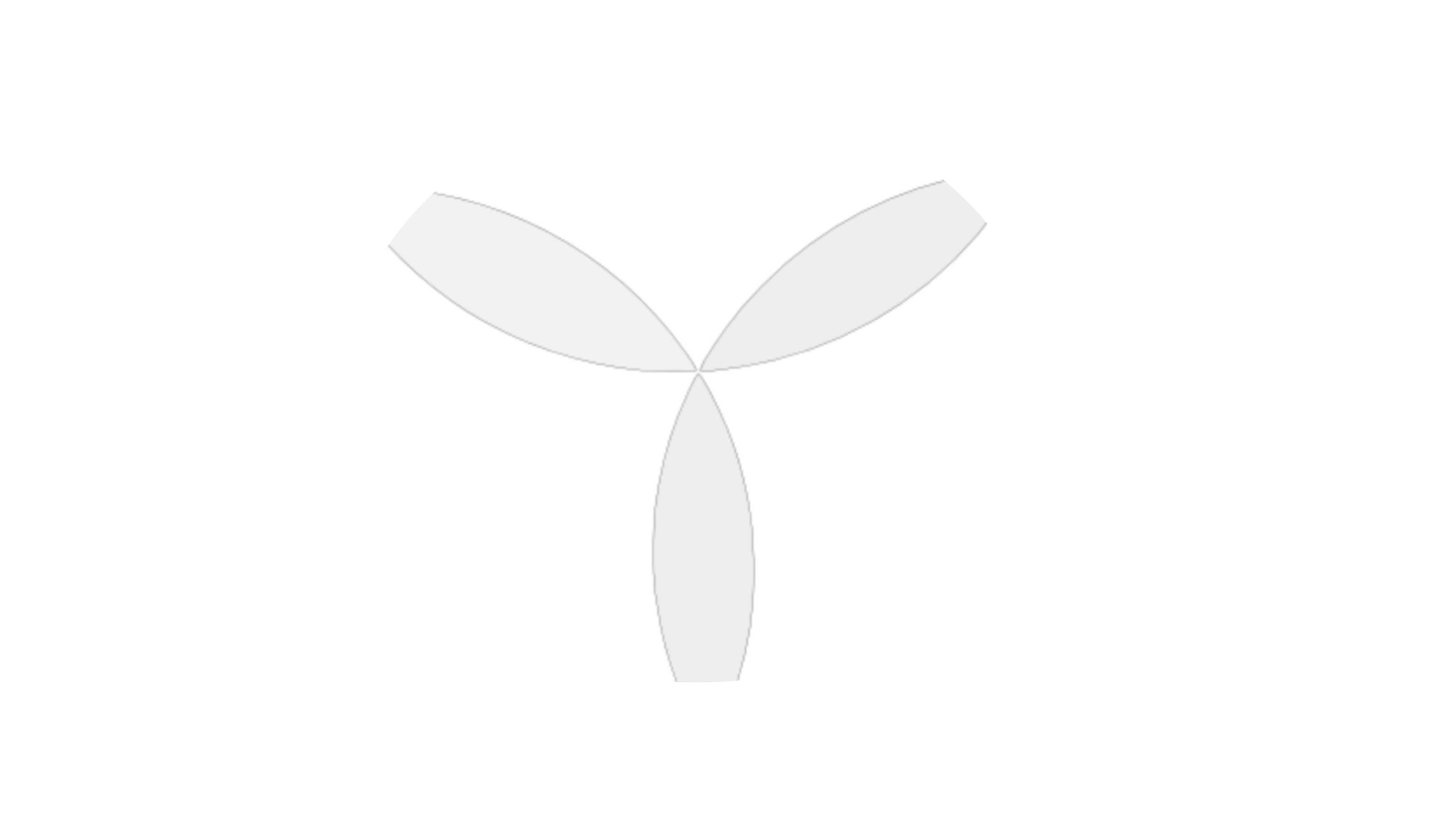} 
\caption{On the left, the cross section of two half cylinders of equal radii touching along a line, an example of (\textbf{a}). The same cross section can be obtained from two half-spheres of equal radii touching tangentially at a point: this is an example of (\textbf{a'}). On the right, the cylinder over the depicted set is an example of (\textbf{b}). The three wedges depicted are defined by drawing three arcs of circles with equal radii with centres at the verteces of an equilateral triangle, and the three arcs pass through a common point. All these examples fulfil, with $g\equiv 0$,
the stationarity and stability conditions (\ref{eq:stationarity}) and (\ref{eq:stability}) for admissible volume-preserving ambient deformations.}
 \label{fig:localstructures}
\end{figure}

The union of two half cylinders with equal radii and touching tangentially along a line (Figure \ref{fig:localstructures}, left) provides an example of the local structure (\textbf{a}) and, assuming $g\equiv 0$, this configuration satisfies (\ref{eq:stationarity}) and (\ref{eq:stability}) for any volume-preserving ambient deformation whose initial velocity is compactly supported in a neighbourhood of a point on the touching line. The stationarity condition is true in view of the fact that cylinders are CMC hypersurfaces with the same mean curvature. The stability condition (\ref{eq:stability}) follows by noticing that each half cylinder is in fact strongly stable, i.e.~stable with respect to $\text{Per}_U(\p C) -\lambda |C\cap U|$ (where $\lambda$ is the mean curvature of the cylinders) for arbitrary ambient deformations; this follows by a well-known semi-calibration argument exploiting the fact that each half-cylinder is a graph (see e.g.~\cite[Appendix B]{BW}). The second variation for an arbitrary ambient deformation, computed for the union of the two half cylinders, can be obtained by adding up the two non-negative contributions computed separately on each half cylinder. This implies the weaker condition (\ref{eq:stability}) for the union of the two half-cylinders.

The union of two half spheres with equal radii and touching tangentially at a point provides an example of the local structure (\textbf{a'}) and, assuming $g\equiv 0$, this configuration satisfies (\ref{eq:stationarity}) and (\ref{eq:stability}) for any ambient deformation whose initial velocity is compactly supported in a neighbourhood of the touching point. The argument is the same as for the half-cylinders, exploiting the fact that each half-sphere is a CMC graph.

For an example of (\textbf{b}) with $N=3$, consider the picture in Figure \ref{fig:localstructures} on the right (more precisely, the product with an interval in $\R^3$). Assuming $g\equiv 0$, this configuration satisfies (\ref{eq:stationarity}) and (\ref{eq:stability}) for any ambient deformation that only acts in a neighbourhood of the central point. The stationarity condition is true since the three cylinders employed are CMC surfaces. For the stability condition, we use once again that the three cylinders are separately strongly stable (since they are graphs). Note that the enclosed volume, for each of the three cylinders taken separately, corresponds to the volume of the portion into which the mean curvature vector points. If we add up the three characteristic functions of these portions, we obtain the constant $1$ plus the characteristic function of the depicted set. Hence, when we consider an ambient deformation (that only acts in a neighbourhood of the centre point), the sum of the three second variations of volume agrees with the second variation of volume of the set in question (changing the functional by a constant term does not affect the second variation). Therefore we can add the three separate second variations of $\text{Per}(\cdot) - \lambda |\cdot \cap U|$ to conclude that the set in question satisfies (\ref{eq:stability}) (in fact, we conclude that the set is strongly stable with respect to $\text{Per}(\cdot) - \lambda |\cdot \cap U|$, for arbitary ambient deformations).

The previous examples show that (\textbf{a}), (\textbf{a'}) and (\textbf{b}) must be permitted in the conclusion of any regularity result that only assumes  (\ref{eq:stationarity}) and (\ref{eq:stability}). As we mentioned, these structures do not reflect observed equilibria,
therefore one needs to identify extra assumptions (in addition to (\ref{eq:stationarity}) and (\ref{eq:stability})) in order to have a meaningful framework that selects observed equilibria. We will discuss this in Section \ref{non_ambient}.

It is an open problem to understand what amount of regularity can be obtained for $\p E$ under assumptions (\ref{eq:stationarity}) and (\ref{eq:stability}) alone, and in principle non-embedded configurations other than (\textbf{a}), (\textbf{a'}) and (\textbf{b}) may be allowed by (\ref{eq:stationarity}) and (\ref{eq:stability}). As we will see in Section \ref{non_ambient}, we do not need to analyse or rule out anything other than the three configurations described, ultimately thanks to the results in \cite{BW}, \cite{BW1}.

\subsection{Non-ambient deformations (and major gain in regularity)}
\label{non_ambient}

To analyse configurations \textbf{(a)}, \textbf{(a')} and \textbf{(b)} and give a mathematical explanation of why they are not observed equilibria, we recall the discussion in Section \ref{heuristic}, where we pointed out that when distinct drops (with regular shapes) touch tangentially, the touching set should be interpreted as an idealization of thin connecting meniscus, and therefore we can consider virtual deformations that smooth out the configuration by thickening (or, similarly, breaking off) the meniscus. These are clearly not ambient deformations and they alter the topology of $\stackrel{\circ}{E}$ or of $\overline{E}$; however, these deformations reflect observed behaviours of liquids, and for this reason we consider them admissible\footnote{Each of the structures \textbf{(a)}, \textbf{(a')} and \textbf{(b)}, although not given by a smoothly embedded hypersurface, is defined by regular enough (at least $C^{1,\alpha}$) pieces, that come together in a well-defined regular fashion. It is this key feature that makes it possible to consider virtual deformations of such structures and ask whether they are physically admissible, because it allows to relate the deformation to observable behaviours of liquid drops. Recall that for an arbitrary Caccioppoli set $E$ the reduced boundary $\p^* E$ is merely a rectifiable set.}.  
Configurations \textbf{(a)} and \textbf{(b)} can be treated in the same way, while \textbf{(a')} has to be treated separately. 

For \textbf{(a)} or \textbf{(b)} we will produce a volume-preserving virtual deformation in $U$ of type ``coalescence'' or ``breakup'' that decreases the energy $\mathcal{E}_U(E)$ to first order and thus violates stationarity. We describe now the key idea, while the formal arguments will be given in Section \ref{C}.

For \textbf{(a)}, denoting by $T$ the curve along which the two disks coincide, we want to consider virtual volume-preserving deformations of $E$ as those depicted in Figure \ref{fig:touching_sing_deform}. From the point of view of computing the first variation, it is irrelevant whether the region between the disks $D_1$ and $D_2$ is in the complement of $E$ or in $E$ (one of these two eventualities must occur): both cases are depicted, in the said order, in Figure \ref{fig:touching_sing_deform}. The deformation in questions changes the topology around $T$ (so it is not an ambient deformation) and corresponds in the former case to the coalescence of the two bulks of liquid; in the latter case it corresponds to a break up of the two bulks at $T$. Such a deformation will act on $\p E \cap U$ by independently smoothing out the cusps of the two degenerate immersions that describe respectively the portion of $\p E \cap U$ to the left of $T$ and to the right of $T$. The mean curvature of these degenerate immersions at $T$ is ``infinite'', due to the presence of a cusp (roughly, the principal curvature in the direction of $T$ is finite, while the one in the direction orthogonal to $T$ is infinite due to the cusp), and points into the cuspidal domain. As the driving gradient of $\mathcal{E}_U (\cdot)$ is given by the mean curvature plus a finite term coming from the potential term, ``thickening the neck'' along $T$ via coalescence or break up decreases the energy to first order (thus violating stationarity).

Passing to the local structure \textbf{(b)}, we note that it corresponds to the case in which $E$ is the union of $N$ ``wedges'' that come together in a $C^{1,\alpha}$ way along a common curve $T$ (at least one wedge has an opening angle at $T$ smaller than $\pi$). The non ambient deformations needed here are those that permit a breakup phenomenon or a coalescence phenomenon at $T$. This can once again be explained by means of first variation arguments, exploiting that the mean curvature vector of each wedge is ``infinite'' at $T$ and points inside the wedge, and employing the deformation that pushes a wedge to its interior around $T$ (or that makes two wedges coalesce around $T$, see Figure \ref{fig:class_sing_deform}).

The previous ideas (partly present in \cite{BW}, \cite{BW1} and detailed in Section \ref{C} in the context addressed here) will guarantee that, as long as the virtual volume-preserving deformations ``coalescence'' and ``break up'' are allowed in the definition of stationarity, structures (\textbf{a}) and (\textbf{b}) are nowhere present in $\p E$. These virtual deformations reflect observed liquid behaviour, so we include them among the admissible ones. 

Then we are in a position to use the following fundamental result  (within the proof of our main theorem). Its first assumption is implied by stationarity for volume-preserving ambient deformations, i.e.~the validity of (\ref{eq:stationarity}), as recalled in Section \ref{ambient}. The validity of (\ref{eq:stationarity}) and (\ref{eq:stability}) implies immediately the validity of the third assumption. The second assumption is implied by the stationarity condition for virtual volume-preserving deformations ``coalescence'' and ``breakup'' discussed above.

\begin{thm}[\cite{BW} and \cite{BW1}]
\label{1}
Let $\Om \subset \R^{3}$ be an open set and let $g$ be an analytic function on $\Om$. Let $E \subset \R^3$ be a Caccioppoli set such that
\begin{itemize}
 \item the first variation of $|\p^*E| \res \Om$ (the multiplicity-$1$ varifold associated to the reduced boundary of $E$) is representable by integration as a vector-valued measure $h\,\mathcal{H}^2 \res (\p^* E \cap \Om)$ with $h \in L^\infty( \mathcal{H}^2 \res (\p^* E \cap \Om);\R^3)$;
 
 \item there exists no open set $U \subset\subset \Om$ such that $\overline{\p^*E} \cap U$ is of the type (\textbf{a}) or (\textbf{b});
 
 \item for every $U \subset \subset \Om$ such that $\overline{\p^*E} \cap U$ is a $C^1$ embedded surface, the set $E$ is stationary with respect to $\mathcal{E}_U$ for volume-preserving ambient deformations that fix $\Om \setminus U$. For every $U \subset \subset \Om$ such that $\overline{\p^*E} \cap U$ is a $C^2$ embedded surface, the set $E$ is stable with respect to $\mathcal{E}_U$ for volume-preserving ambient deformations that fix $\Om \setminus U$.
\end{itemize}
Then  $\overline{\p^* E} \cap \Om$ is smoothly embedded (also analytic) away from a (possibly empty) set of points that are isolated in $\Om$, i.e.~for every such point $p$ there exists a neighbourhood $U_p$ contained in $\Om$ such that $(\overline{\p^* E} \setminus \{p\}) \cap U_p$ is smoothly embedded. Locally around any such $p$, the structure of $\overline{\p^*E}$ is the one in \textbf{(a')}.
\end{thm}

The above theorem is a special case of the results in \cite{BW} and \cite{BW1} (see also Appendix \ref{B}) and will be applied, in our case, to every connected component of $\overline{E}\cap \Om$, where $E$ is the Lebesgue representative of the Caccioppoli set that denotes the region occupied by the liquid.
Theorem \ref{1} provides the first (and major) gain in regularity, giving a very precise characterisation of $\p E$; the leftover lack of embeddedness is confined to isolated points, where the only allowed structure is \textbf{(a')}.

\begin{oss}
The difficult conclusion, in Theorem \ref{1}, is the $C^2$ embeddedness away from isolated points at which the structure is \textbf{(a')}. From there, higher regularity and analyticity of the surface are immediate by elliptic PDE theory, since $g$ is analytic.
\end{oss}

The structure \textbf{(a')} is also not observed for liquids in stable equilibrium. The expectation should be that, considering again the touching point as the idealization of a thin connecting neck, the possibility of coalescence provides a virtual deformation that violates stable equilibrium. This may appear, at first sight, to be very similar to the situation discussed for \textbf{(a)}. Indeed the fact that a coalescence phenomenon should be energy convenient (when we have a tiny meniscus) is sometimes attributed, in the literature, to the large amount of curvature present in the meniscus (along the lines of the explanation given when we ruled out (\textbf{a})). However, upon close inspection, 
that turns out not to completely accurate. It is not immediately clear why the curvature should play a role, since $\mathcal{E}_{\Om}$ selects the mean curvature (and not the full curvature) as the key quantity that drives towards an equilibrium. The mean curvature of a tiny neck, obtained by a smoothing of two touching spherical caps, may attain any value, in particular the neck could be part of a Delaunay surface, which is stationary for $g\equiv 0$ (in stark contrast with \textbf{(a)}, where smoothing leads to very high mean curvature)\footnote{For further comments on bounds on the full curvature, see Section \ref{droplets}.}. 

Our next task is then to understand which assumption can rule out \textbf{(a')}. It turns out that, allowing virtual volume-preserving deformations ``coalescence'' and ``break up'' does indeed permit to rule out \textbf{(a')}; unlike for \textbf{(a)} and \textbf{(b)}, however, it will be the stability assumption, rather than stationarity, that will do so (stationarity alone would not suffice). In fact we will prove the following proposition in Section \ref{finalproof} (the formal notions of ``coalescence'' and ``break-up'' deformations are given in Definitions \ref{Dfi:coalescence_a}, \ref{Dfi:coalescence_a'}, \ref{Dfi:coalescence_b} and the notions of stationarity and stability for these deformations are given in Definition \ref{Dfi:station_stable_one_sided} below):

\begin{Prop}
\label{Prop:a'}
Let $U$ and $E$ be as in the definition of \textbf{(a')}. 
Assume that $E$ is stationary in $U$ for ambient volume-preserving deformations, with respect to $\mathcal{E}_U$. Then $E$ is not stable in $U$ with respect to $\mathcal{E}_U$ for volume-preserving deformations of ``coalescence'' and ``break up'' type.
\end{Prop}

Applying the Proposition locally around every point of non-embeddedness in the conclusion of Theorem \ref{1} we will complete the proof of the embeddedness of $\p E \cap \Om$, and thus of the main theorem (Theorem \ref{thm:main}).

\subsection{Stationarity and stability for ``coalescence'' and ``break-up'' virtual deformations: formal definitions and statement of main result}

Having discussed the need for the ``coalescence'' and ``break-up'' virtual deformations and their role in achieving the embeddedness conclusion, we will now formally define these deformations. As we found out, we only need to define them for very specific starting configurations, namely \textbf{(a)}, \textbf{(a')}, \textbf{(b)}. We also note that the deformation of $\p^* E$ is actually the same for ``coalescence'' and ``break-up'': we speak of coalescence if distinct connected components of the interior of $E$ merge and we speak of break-up if a connected component of the closure of $E$ gets split. Passing to the complement of $E$ we turn ``coalescence'' into ``break-up'' and vice-versa (see Figures \ref{fig:touching_sing_deform}, \ref{fig:class_sing_deform}).

\medskip

With reference to the structures \textbf{(a)}, \textbf{(a')}, \textbf{(b)} described in Section \ref{non_ambient}, the definition of ``coalescence'' or ``break-up'' deformations will be given after choosing a convenient system of coordinates and a sufficiently small neighbourhood.

\medskip

\noindent \textbf{(a)} Upon choosing coordinates suitably and making $U$ smaller if needed, we assume that $0\in T$, $D_1$ and $D_2$ are graphs over the plane spanned by the first two coordinates and this plane is tangent to $D_1$ and $D_2$ at $0$, $D_1$ is below $D_2$, $T$ is a connected curve and it is a graph over the second coordinate axis. Let $(\gamma_1(s), s, \gamma_3(s))$ be a smooth injective parametrization of $T$. Denote the four connected components of $\p E \cap U \setminus T$ by $A_1, A_2, A_3, A_4$, where $A_1, A_2$ are contained in the top embedded disk and $A_3, A_4$ in the bottom embedded disk, with $A_1, A_3$ defined by the condition $y\in A_1 \cup A_3 \Rightarrow y_1<\gamma_1(y_2)$ (i.e.~$A_1, A_3$ are ``on the left of $T$'') and $A_2, A_4$ defined by the condition $y\in A_2 \cup A_4 \Rightarrow y_1>\gamma_1(y_2)$ (i.e.~$A_1, A_3$ are ``on the right of $T$''). Then $A_1 \cup T \cup A_3$ and $A_2 \cup T \cup A_4$ are topologically disks; both fail to be smoothly embedded or immersed exactly at points in $T$ ($T$ is a cusp). 
We then consider two disks $D_L$ and $D_R$ and a $C^1$ proper map $\psi_{LR}:D_L \cup D_R \to U$ such that its restriction to $D_L$ and $D_R$ has image respectively equal to $A_1 \cup T \cup A_3$ and $A_2 \cup T \cup A_4$; moreover, $\psi_{LR}$ is a smooth embedding away from the inverse image of $T$ (where the differential becomes degenerate --- note that $\psi_{LR}^{-1}(T)$ disconnects both $D_L$ and $D_R$ and these are the only points of non-injectivity).

\begin{Dfi}[coalescence/break up from \textbf{(a)}; degenerate immersions of two disks]
\label{Dfi:coalescence_a}
We say that $\Psi_{LR}(t,x):[0,\eps) \times (D_L \cup D_R) \to U$ (with $\eps>0$) is a ``coalescence'' or ``break-up'' deformation in $U$ if it is $C^1$ as a function of $(t,x)$, it is $C^2$ as a function of $t$, $\Psi_{LR}(0,\cdot)=\psi_{LR}(\cdot)$, and if there exists an open non-empty ball $B \subset \subset U$ that contains $0$ and such that, writing $B_{LR}=\psi_{LR}^{-1}(B)$, and writing $\psi_{LR}^t: D_L \cup D_R \to U$ for the map defined by $\psi_{LR}^t(\cdot) = \Psi_{LR}(t,\cdot)$, we have:
for all $t\in [0,\eps)$, the map $\psi_{LR}^t$ is an embedding when restricted to $B_{LR} \setminus \psi_{LR}^{-1}(T)$; for all $t\in [0,\eps)$, $\psi_{LR}^t= \psi_{LR}$ in $(D_L \cup D_R)\setminus B_{LR}$.
\end{Dfi}

\noindent \textbf{(a')} Upon choosing coordinates suitably and making $U$ smaller if needed, we assume that $\{0\}=D_1 \cap D_2$, $D_1$ and $D_2$ are graphs over the plane spanned by the first two coordinates and this plane is tangent to $D_1$ and $D_2$ at $0$. We describe the initial configuration of $\p E \cap U$ as the image of a $C^1$ proper map $\psi_C: S^{1} \times (-1,1) \to U$, with $\psi_C$ an embedding away from $S^{1} \times \{0\}$ and with $S^1 \times \{0\}$ mapped to the touching point of the two embedded disks, that we assume to be $0\in\R^3$.\footnote{E.g.~$(e^{i \theta}, z)\in S^1 \times (-1,1) \to (z^2 \cos \theta, z^2 \sin \theta, \text{sgn}(z) z^4)\in\R^3$ covers two touching paraboloids.} The differential of $\psi_C$ is degenerate (and $\psi_C$ becomes non-injective) on $S^{1} \times \{0\}$.

\begin{Dfi}[coalescence/break up from \textbf{(a')}; degenerate immersion of a cylinder]
\label{Dfi:coalescence_a'}
We say that $\Psi_{C}(t,x):[0,\eps) \times (S^{1} \times (-1,1)) \to U$ (with $\eps>0$) is a ``coalescence'' or ``break-up'' deformation in $U$ if it is $C^1$ as a function of $(t,x)$, it is $C^2$ as a function of $t$, $\Psi_{C}(0,\cdot)=\psi_{C}(\cdot)$ and if there exists an open non-empty ball $B \subset \subset U$ that contains $0$ and such that, writing $B_{C}=\psi_{C}^{-1}(B)$, and writing 
$\psi^t_{C}:S^{1} \times (-1,1) \to U$ for the map defined by $\psi^t_{C}(\cdot)=\Psi_C(0,\cdot)$, we have: for every $t \in [0,\eps)$, $\psi^t_{C}$ restricted to $S^1\times ((-1,0) \cup (0,1))$ is an embedding; for all $t\in[0,\eps)$, $\psi_{C}^t= \psi_{C}$ in $(S^{1} \times (-1,1))\setminus B_{C}$.
\end{Dfi}

\noindent \textbf{(b)} We set coordinates so that $0\in T$ and choose $U$ sufficiently small so that $T$ is connected. We assume that the indexation of $L_j$ is such that $L_{j}$ and $L_{j+1}$ are adjacent. We describe $\p E \cap U = \cup_{j=1}^{2N} L_j$ as the image of a $C^{1}$ proper map $\psi_{cl}: D_1 \cup \ldots \cup D_N \to U$, where $D_j$ are disks and $\psi_{cl}$ is an embedding away from the inverse image of $T$ (which disconnects each disk $D_j$), $\psi_{cl}(D_i) = L_j \cup L_{j+1}$ for some $j$ (i.e.~each disk covers two adjacent hypersurfaces-with-boundary) and the differential of $\psi_{cl}$ is non-injective exactly on the inverse image of $T$.

\begin{Dfi}[coalescence/break up from \textbf{(b)}; degenerate immersion of $N$ disks.]
\label{Dfi:coalescence_b}
We say that $\Psi_{cl}(t,x):[0,\eps) \times (D_1 \cup \ldots \cup D_N) \to U$ (with $\eps>0$) is a ``coalescence'' or ``break-up'' deformation in $U$ if it is $C^1$ as a function of $(t,x)$, it is $C^2$ as a function of $t$, $\Psi_{C}(0,\cdot)=\psi_{C}(\cdot)$ and if there exists an open non-empty ball $B \subset \subset U$ that contains $0$ and such that, writing $B_{cl}=\psi_{cl}^{-1}(B)$, and writing 
$\psi^t_{cl}:D_1 \cup \ldots \cup D_N \to U$ for the map defined by $\psi^t_{cl}(\cdot)=\Psi(t,\cdot)$, we have: for every $t\in[0,\eps)$, $\psi^t_{cl}$ is a smooth embedding when restricted to $B_{LR}\setminus \psi_{cl}^{-1}(T)$; for all $t\in[0,\eps)$ we have $\psi_{cl}^t= \psi_{cl}$ in $(D_1 \cup \cdots \cup D_N)\setminus B_{LR}$.
\end{Dfi}

\textit{Initial speed}. Denoting by $\psi^t$ the deformation in either of the three cases, the initial velocity $\vec{v}:=\left.\frac{\p}{\p t}\right|_{t=0^+} \psi^t$ is a $C^1_c$ vector field on the domain of $\psi$ (i.e.~a $C^1_c$ and $\R^3$-valued function on the domain of $\psi^0$). In particular, we require, in all three cases, $\vec{v}$ to be bounded (in other words, we require the virtual deformations to happen at bounded speed, in accordance with a concrete deformation of a liquid.) We note that $\vec{v}$ fails to be identifiable with an ambient vector field on $\p E \cap U$ exactly at $T$, in all three cases. 

\medskip

\textit{Volume-preserving constraint for coalescence and break up virtual deformations}. Note that in all three cases above, the image of $\psi^t$ is the (topological) boundary of a Caccioppoli set $E_t$ in $U$, for all $t$. In fact, choosing two alternative orientations, it is the boundary of two Caccioppoli sets, one being the complement of the other. In case \textbf{(a)}, for example, the interior of $E=E_0$ could be either the portion of $U$ ``between'' the two embedded disks (that touch each other tangentially at $T$), or the complement of this region. We do not need to distinguish the two possibilities, we merely need to define $E_t$ coherently with $E_0=E$, that is ensuring that $\chi_{E_t}$ is continuous in $L^1(U)$. Depending on which set is the interior of $E$, the same deformation $\psi^t$ of $\p E$ in $U$ could be ``coalescence'' or ``break-up'' (see e.g.~Figure \ref{fig:touching_sing_deform}). We will say that the deformation $\psi^t$ is volume-preserving if $|E_t \cap U|$ is constant in $t$; note that this does not depend on which set is $E$ and which is $U\setminus E$, since preserving the volume of $E \cap U$ is equivalent to preserving the volume of $U\setminus E$.

\begin{Dfi}
\label{Dfi:station_stable_one_sided}
 Let $E$ be a Caccioppoli set in $\R^3$ with $\p E \cap U$ given by one of the three structures \textbf{(a)}, \textbf{(a')}, \textbf{(b)}. Stationarity of $E$ in $U$ for volume-preserving (one-sided) ``coalescence'' or ``break-up'' deformations means that for any $\psi^t$ as in $\psi_{LR}^t$ of Definitions \ref{Dfi:coalescence_a}, or as in $\psi_C^t$ of \ref{Dfi:coalescence_a'}, or as in $\psi_{cl}^t$ of \ref{Dfi:coalescence_b}, and with $\psi_t$ volume-preserving in $U$ (and with bounded initial speed), the energy $\mathcal{E}_U$ does not decrease to first order, i.e.
 $$\left.\frac{d}{dt}\right|_{t=0^+} \mathcal{E}_U(E_t) \geq 0.$$
 Similarly, stability of $E$ in $U$ for ``coalescence'' or ``break-up'' volume-preserving deformation means that (for the same class of $\psi^t$) the second variation (at $t=0^+$) is non-negative if the first variation (at $t=0^+$) vanishes, that is: for any $\psi^t$ as in Definitions \ref{Dfi:coalescence_a}, \ref{Dfi:coalescence_a'}, \ref{Dfi:coalescence_b} that is volume-preserving and has bounded initial speed we have
 $$\left.\frac{d}{dt}\right|_{t=0^+} \mathcal{E}_U(E_t) = 0 \Rightarrow \left.\frac{d^2}{dt^2}\right|_{t=0^+} \mathcal{E}_U(E_t) \geq 0 .$$
\end{Dfi}

\begin{oss}
 The stationarity condition in Definition \ref{Dfi:station_stable_one_sided} is an inequality because the deformations are one-sided ($t\in [0,\eps)$) and therefore the natural stationarity condition requires that the deformation ``does not decrease the energy to first order''. If the (one-sided) first variation is strictly positive, then the non-negativity of the second variation would be an unnatural requirement, hence the formulation of stability given. 
\end{oss}

We can now summarise and define precisely the hypotheses that we make to describe a condition of stable equilibrium for a liquid that occupies the region $E$.

\begin{Dfi}
\label{Dfi:station_stable_full}
Let $E$ be a Caccioppoli set in $\R^3$ and let $\Om \subset \R^3$ be an open set. We say that $E$ is stationary and stable in $\Om$ with respect to the energy $\mathcal{E}_{\Om}$ if the following two conditions hold:

\noindent (i) For every admissible (as in Definition \ref{Dfi:vol-pres_Leb}) volume-preserving ambient deformations $\psi_t$, with initial speed compactly supported in $\Om$, (\ref{eq:stationarity}) and (\ref{eq:stability}) hold.

\noindent (ii) Let $U\subset \subset \Om$ be such that $\p E \cap U$ is given by one of the three structures \textbf{(a)}, \textbf{(a')}, \textbf{(b)}. Then the stationarity and stability inequalities in Definition \ref{Dfi:station_stable_one_sided} hold for any choice of ``coalescence'' or ``break-up'' volume-preserving deformation in $U$.
\end{Dfi}

We can now state our main result in precise form.

\begin{thm}
 \label{thm:main}
Let $\Om$ be an open set in $\R^{3}$, $g:\Om\to \R$ analytic and let $E$ be a set of locally finite perimeter (Caccioppoli set) in $\R^3$. Assume that $E$ is stationary and stable in $\Om$ with respect to $\mathcal{E}_{\Om}$ in the sense of Definition \ref{Dfi:station_stable_full}. Then $\overline{\p^* E} \cap \Om = \p^* E \cap \Om$ is a smoothly embedded (analytic) surface. (Moveover, the mean curvature vector of $\p E \cap U$ is given by $(g-\lambda)\nu_E$ for some $\lambda \in \R$ that possibly depends on the connected component of $E$, with $\nu_E$ denoting the unit outer normal to $E$).
\end{thm}

\noindent \textit{Summary of the proof of Theorem \ref{thm:main}.} The proof of Theorem \ref{thm:main} (informally outlined in Sections \ref{ambient} and \ref{non_ambient}) will be given in the next three sections. In Section \ref{A} we verify the claim that the first variation of $|\p^* E|\res \Om$ is in $L^\infty$. In Section \ref{C} we show that structures \textbf{(a)} and \textbf{(b)} are prevented by the stationarity assumption in (ii) of Definition \ref{Dfi:station_stable_full}. These two facts, together with condition (i) of Definition \ref{Dfi:station_stable_full}, verify the validity of the hypotheses of Theorem \ref{1}, whose conclusion reduces the proof of Theorem \ref{thm:main} to the analysis of structure \textbf{(a')}. Structure \textbf{(a')} will be ruled out in Section \ref{finalproof} thanks to the stability assumption in (ii) of Definition \ref{Dfi:station_stable_full}. 

\begin{oss}[\textit{curvature bounds}]
In Section \ref{droplets} we will add, to the conclusion of Theorem \ref{thm:main}, some remarks on the boundedness properties of $|A_{\p E}|$.
\end{oss}

\begin{oss}[\textit{the case of minimisers}]
 \label{oss:minimisers}
The case of energy-minimisers has a long and fruitful history. The local minimising conditions for $E$ in an open set $U$ entails a comparison with Caccioppoli sets that coincide with $E$ in the complement of $U$ and are $L^1$-close to $E$. Such a condition easily implies the assumptions made in Theorem \ref{thm:main}, hence Theorem \ref{thm:main} recovers the regularity conclusions known for minimisers (see e.g.~\cite{GonzMassTaman}, \cite{GonzMassTaman2}, \cite{Maggi}). We point out that, under a minimising assumption, density estimates analogous to those in \cite[Theorem 16.14]{Maggi} play a key role in the regularity theory. The proof of these density estimates, however, makes use of the comparison argument with a ``competitor'' Caccioppoli set that is $L^1$-close to $E$. The competitor is built by (roughly speaking) cutting off the part of $E$ contained in a small ambient ball (and balancing the volume constraint somewhere else). This operation does not give rise to a deformation that can be related to a concrete liquid movement, since $\p^* E$ is of unknown structure in said ball. The spirit in the choices of variational assuptions in Theorem \ref{thm:main} is to have virtual deformations that can be thought of as concretely replicable.
\end{oss}

\begin{oss}
The variational assumptions of Theorem \ref{thm:main} can be weakened as explained in Appendix \ref{B}.
\end{oss}

\section{Volume-constraint as Lagrange multiplier}
\label{A}

We prove the claim in Section \ref{ambient}, i.e.~that (\ref{eq:stationarity}) leads to the $L^\infty$ condition on the first variation and to the fact that the generalised mean curvature is $(g-\lambda)\nu$.

Let $E$ be a Caccioppoli set in $\R^{n+1}$ and $U$ be an open set. Recall that (see e.g.~\cite[Theorem 17.5 and Proposition 17.8]{Maggi}), for an ambient deformation $\psi_t$ in $U$ with initial velocity $X \in C^1_c(U;\R^{n+1})$, the first variation of perimeter in $U$ is given by $\int_{\p^* E} \text{div}_{\p^* E} X$ and the first variation of $\int_E g$ is given by $\int_{\p^* E} g X \cdot \nu$, where $\nu$ is the outer unit normal to $\p^* E$. In particular, the first variation of $\int_E 1 = |E|$ is given by $\int_{\p^* E} X \cdot \nu$. These first variations are thus independent of the particular deformation, they only depend on the initial velocity. We have the following results.

\begin{Prop}
\label{Prop:Lagr_mult}
Let $E$ be a Caccioppoli set in $\R^{n+1}$ and $U\subset \R^{n+1}$ an open set. The following are equivalent:

\noindent $\bullet$ $E$ is stationary in $U$ with respect to the functional $\mathcal{E}_U(\cdot)$ for volume-constrained ambient deformations in $U$;

\noindent $\bullet$ there exists $\lambda \in \R$ such that $E$ is stationary in $U$ with respect to the functional $\mathcal{E}_U(\cdot) - \lambda |\cdot \cap U|$ for \textit{arbitrary} ambient deformations in $U$. 
\end{Prop}

An analogous equivalence is proved in \cite[Theorem 17.20]{Maggi} (in the case $g\equiv 0$, however the argument extends to arbitrary $g$) for the more restrictive class of volume-constrained minimizers ``among Caccioppoli sets'', or (with the same proof) for volume-constrained local minimizers, i.e.~for Caccioppoli sets $E$ that minimize the perimeter among Caccioppoli sets of the same volume that are close to $E$ in the $L^1$-topology. Since we are interested in stationarity for ambient deformations (which are more restrictive than deformations with respect to the $L^1$-topology), we adapt the arguments of \cite{Maggi} in the proof below, adapting in step 1 also an argument from \cite[Lemma 2.4]{BarbDoCarmo}.

\begin{proof}
\textit{Step 1}. Let $X \in C^1_c(U;\R^{n+1})$ be such that $\int_{\p^* E} X \cdot \nu =0$. Then there exists an ambient deformation $\psi_t$ of $E$ that is volume-preserving in $U$ and has initial velocity $\left.\frac{d \psi_t}{dt}\right|_{t=0}  = X$. To see this, choose any $Y\in C^1_c(U;\R^{n+1})$ such that $\int_{\p^* E} Y \cdot \nu \neq 0$. Consider the $2$-parameter family of diffeomorphisms $Id +t X +s Y$, for $s, t \in (-\eps, \eps)$, with $\eps>0$ sufficiently small.
The function $(t,s) \to |(Id +t X +s Y)(E) \cap U|$ is of class $C^1((-\eps, \eps)\times (-\eps,\eps))$. Indeed, the value of the volume at $(t,s)$ is given by $\int_U \chi_E \circ (Id+tX +sY)^{-1}$, where $\chi_E$ is the characteristic function of $E$; changing variable and computing the Jacobian of the diffeomorphism $x\to x+tX(x)+sY(x)$, we can rewrite the volume at $(t,s)$ as $\int_U \chi_E (1+t \,\text{div}X +s\, \text{div}Y + O(s^2) + O(t^2))$, with $O(s^2), O(t^2)$ continuous in $x$ and smooth in $t,s$. 
This also implies that the partial derivative $\left.\frac{\p}{\p s}\right|_{(t,s)=(0,0)}|(Id +t X +s Y)(E) \cap U|$ is $\int_U \chi_E \,\text{div}Y = \int_{\p^* E} Y \cdot \nu \neq 0$ (by definition of $\p^* E$). By the implicit function theorem, the level set $\{(t,s):|(Id +t X +s Y)(E) \cap U| = |E\cap U|\}$ is of the form $\{(t, s(t)\}$ for some $C^1$ function $s(t)$ defined on a possibly smaller interval $t\in (-\eps_0, \eps_0)$; moreover, $s'(0)=\frac{-\int_{\p^* E} X \cdot \nu}{\int_{\p^* E} Y \cdot \nu} =0$. We thus obtain that $\psi_t=Id +tX +s(t) Y$ is a volume preserving ambient deformation in $U$ with initial velocity $X + s'(0) Y =X$. 

\medskip

\textit{Step 2}. Let us prove (in steps 2, 3, 4) that the first condition of the Proposition implies the second. Given $X$ as in step 1, the stationarity assumption on $E$ gives that the first variation of $\mathcal{E}_U$ is $0$ along the deformation $\psi_t$ exhibited in step 1. The expression of the first variation of $\mathcal{E}_U$ agrees with the one computed along any ambient deformation (not necessarily volume-preserving) with initial velocity $X$, as recalled before the statement of the Proposition. We have therefore proved that, for $X \in C^1_c(U;\R^{n+1})$ such that $\int_{\p^* E} X \cdot \nu =0$, the first stationarity condition in the Proposition implies that $\int_{\p^* E} \text{div}_{\p^* E} X+\int_{\p^* E} g X \cdot \nu=0$. 

Note that this identity is also obtained from the second stationarity condition in the proposition when employing a deformation with initial speed the given $X$, regardless of the choice of $\lambda$ (because the first variation of the term $\lambda |\cdot \cap U|$ is $\lambda \int_{\p^* E} X\cdot \nu$, which vanishes for the given $X$).

\medskip

\textit{Step 3}. We prove the following. Assume that $E$ satisfies the first stationarity condition in the proposition. Let $X,Y \in C^1_c(U;\R^{n+1})$ be such that $\int_{\p^* E} X \cdot \nu \neq 0$, $\int_{\p^* E} Y \cdot \nu \neq 0$. Then the ratios $\frac{\int_{\p^* E} \text{div}_{\p^* E} X + \int_{\p^* E} g X \cdot \nu}{\int_{\p^* E} X \cdot \nu}$ and $\frac{\int_{\p^* E} \text{div}_{\p^* E} Y+ \int_{\p^* E} g Y \cdot \nu}{\int_{\p^* E} Y \cdot \nu}$ are equal.

Define the vector field $T=X - \frac{\int_{\p^* E} X \cdot \nu}{\int_{\p^* E} Y \cdot \nu} Y$. Then $\int_{\p^* E} T \cdot \nu = 0$, and $T\in C^1_c(U;\R^{n+1})$. By step 2 we have that the first variation of $\mathcal{E}_U$ along any deformation with initial velocity $T$ is $0$, i.e.~$\int_{\p^* E} \text{div}_{\p^* E} T + \int_{\p^* E} g T \cdot \nu=0$. Substituting for $T$, the conclusion follows.

\medskip

\textit{Step 4}. We  have, by step 3, that there exists $\lambda \in \R$ such that $ \int_{\p^* E} \text{div}_{\p^* E} X + \int_{\p^* E} g X \cdot \nu - \lambda \int_{\p^* E}X \cdot \nu = 0$ for every $X$ such that $\int_{\p^* E} X \cdot \nu \neq 0$. The same identity holds for $X$ with  $\int_{\p^* E} X \cdot \nu = 0$ as well, as proved in step 2. This amounts to the stationarity of $E$ for any ambient deformations with respect to the functional $\mathcal{E}_U - \lambda |E \cap U|$. We have thus proved that the first stationarity condition in the proposition implies the second.

\medskip

\textit{Step 5}. The second stationarity condition in the proposition easily implies the first, upon noticing that for any volume preserving ambient deformation $\psi_t$, its initial velocity $\left.\frac{d}{dt}\right|_{t=0}\psi_t$ satisfies $\int_{\p^* E} \left(\left.\frac{d}{dt}\right|_{t=0}\psi_t\right) \cdot \nu =0$. The second stationarity condition reduces to the first, for such a vector field.
\end{proof}

As a consequence (see e.g.~\cite{Maggi} or \cite[Remark 2.19]{BW}) we get the following:

\begin{Prop}
\label{Prop:Lagr_mult_mean_curv}
For $\lambda \in \R$, let $E$ be stationary with respect to $\mathcal{E}_U(\cdot) - \lambda |\cdot \cap U|$ for arbitrary ambient deformations in $U$. Then the first variation  (in the varifold sense) of $|\p ^* E|\res U$ is in $L^\infty$ with respect to $\Hc^n \res (\p^*E\cap U)$, and its density, i.e.~the generalized mean curvature of $\p ^* E$ in $U$, is given by $(g-\lambda)\nu$, where $\nu$ is the measure-theoretic unit outer normal to $E$ in $U$.
\end{Prop}

\section{Ruling out (a) and (b) by first variation}
\label{C}
In this section we will show that configurations \textbf{(a)} and \textbf{(b)} are ruled out by the stationarity assumption of Theorem \ref{thm:main}, specifically by the stationarity requirement in Definition \ref{Dfi:station_stable_one_sided}.

\medskip

In case \textbf{(a)}, choosing $U$ and the coordinate system as in Section \ref{assumptions}, we denote by $T$ the $1$-submanifold (connected curve) of coincidence and assume that $D_1, D_2$ are graphs of smooth functions $u_1\leq u_2$ over a common plane, that is tangent to both graphs at $0\in T$, and with $D_1\cap D_2=T$. The set $E$ is either (case 1) the union of the subgraph of $u_1$ and the supergraph of $u_2$ (as in the top left picture in Figure \ref{fig:touching_sing_deform}), or (case 2) it is the set of points that lies between the two graphs (as in the bottom left picture in Figure \ref{fig:touching_sing_deform}); in either case, $\overline{E}$ is connected in the neighbourhood $U$ that we are analysing and therefore (Proposition \ref{Prop:Lagr_mult}) stationary, for arbitrary ambient deformations, with respect to $\mathcal{E}_U(\cdot)-\lambda|\cdot \cap U|$.  We note that $E$ is stationary in $U$ with respect to $\text{Per}_U(\cdot) + \int_{\cdot \cap U} g -\lambda |\cdot \cap U|$ if and only if $\R^3 \setminus E$ is stationary in $U$ with respect to $\text{Per}_U(\cdot) - \int_{\cdot \cap U} g +\lambda |\cdot \cap U|$. Therefore, upon redefining $\lambda$ and $g$ by a change of sign, we can (and do) assume that we are in case 2.

\begin{figure}[h]
\centering
 \includegraphics[width=6cm]{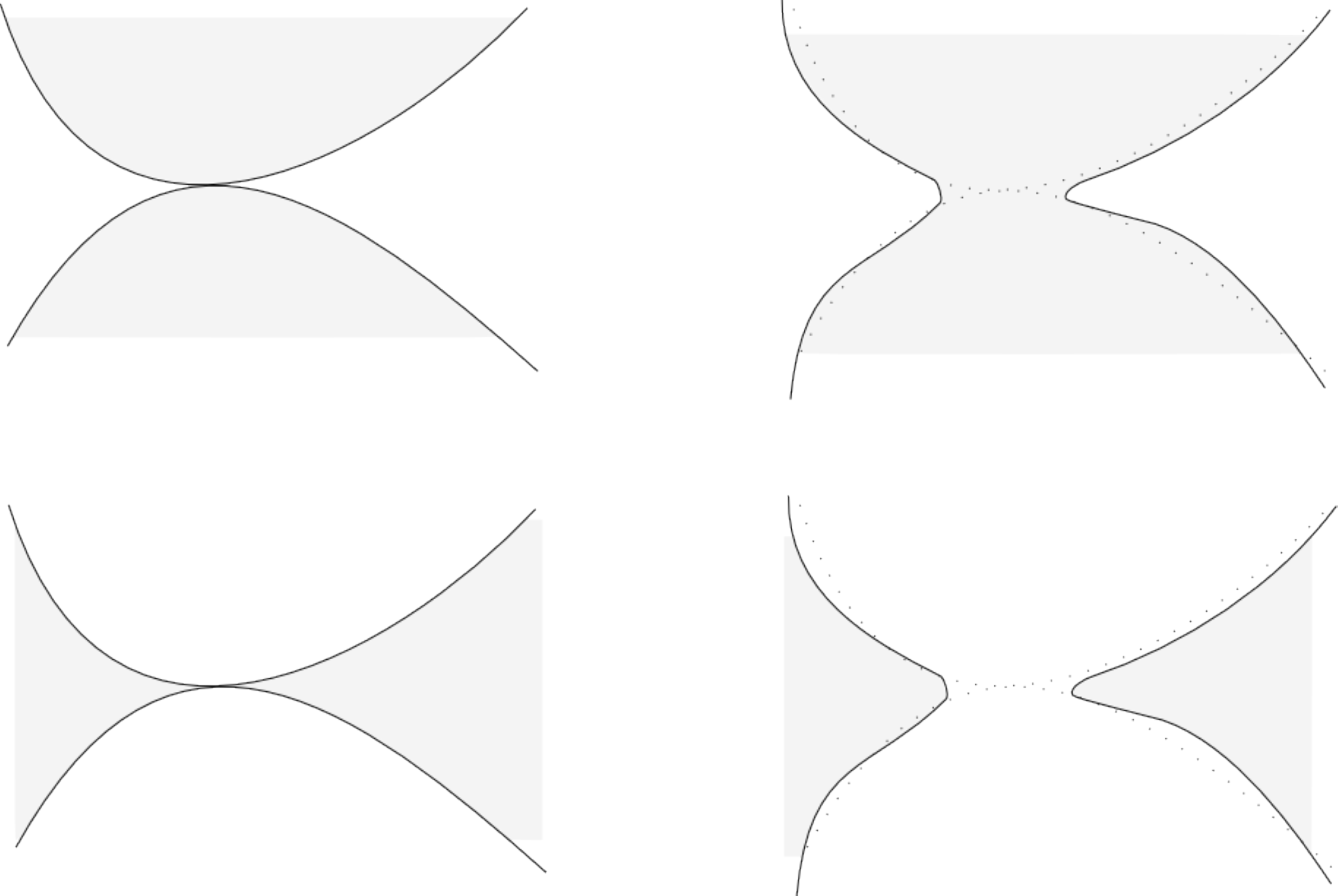} 
\caption{Top row: virtual deformation ``coalescence''. Bottom row: virtual deformation ``breakup''. The boundary is deformed in the same way in both cases. The picture depicts cross sections with starting configuration \textbf{(a)}, in which case one should imagine cylinders over the given figures; alternatively, the same cross sections could be obtained with starting configuration \textbf{(a')}, as in the example of two spherical caps touching at a point.}
 \label{fig:touching_sing_deform}
\end{figure}

Stationarity with respect to $\text{Per}_U(\cdot) + \int_{\cdot \cap U} g -\lambda |\cdot \cap U|$ for ambient deformations implies that $\p E \setminus T$ is smooth and has mean curvature $(g -\lambda)\nu$, for some $\lambda \in \R$, where $\nu$ is the outer unit normal to $E$.\footnote{It is easily checked via the maximum principle, see e.g.~\cite{BW}, \cite{BW1}, that the mean curvature vector of $\p E$ points to the exterior of $E$ in case 2, in a neighbourhood of $T$, hence one must have $\lambda < g$ in a neighbourhood of $T$ --- we do need to use this fact.}
Let $\p E = \Gamma_1 \cup \Gamma_2$ and $E= F_1\cup F_2$ be the decompositions such that $F_1, F_2$ are open sets with $\p F_j = \Gamma_j$ for $j\in\{1,2\}$ and $F_1\cap F_2 = \emptyset$, $\overline{F}_1\cap \overline{F}_2 =T$; here $\Gamma_1 = A_1 \cup A_3 \cup T$ and $\Gamma_2 = A_2 \cup A_4 \cup T$, with notations as in Section \ref{assumptions}. (With reference to the bottom left picture of Figure \ref{fig:touching_sing_deform} $F_1$ and $F_2$ are the two cuspidal domains repectively on the left and on the right of the touching curve $T$.) The set $F_1\cup F_2$ is stationary for arbitrary ambient deformations with respect to $\text{Per}_U(\cdot) + \int_{\cdot \cap U} g -\lambda |\cdot \cap U|$ for a choice of $\lambda$ and $g$. We will show, in a first instance, that however neither $F_1$ nor $F_2$ is (separately) stationary, for volume-preserving ambient deformations, with respect to $\text{Per}_U(\cdot) + \int_{\cdot \cap U} g$. 

The condition that the mean curvature on the smooth parts of $\Gamma_j$ is $(g-\lambda)\nu$ (where $\nu$ is the outer unit normal to $F_k$) implies already that, for each $k\in\{1, 2\}$, $F_k$ is not stationary for arbitrary ambient deformations with respect to $\text{Per}_U(\cdot) + \int_{\cdot \cap U} g -\lambda' |\cdot \cap U|$ if $\lambda' \neq \lambda$. We therefore only need to prove that, for each $k\in\{1, 2\}$, $F_k$ is not stationary for arbitrary ambient deformations with respect to $\text{Per}_U(\cdot) + \int_{\cdot \cap U} g -\lambda |\cdot \cap U|$. Then Proposition \ref{Prop:Lagr_mult} will imply that $F_k$ is not stationary, for volume-preserving ambient deformations, with respect to $\text{Per}_U(\cdot) + \int_{\cdot \cap U} g$. 

We analyse the first variation of area for the integral varifold $|\p^* F_k|$, for each chosen $k\in \{1,2\}$. This varifold is given by the sum of (the multiplicity-$1$ varifolds associated to) two smooth sufaces-with-boundary (the boundary is $T$ for both). The first variation is additive, so we compute it separately for the two sufaces-with-boundary and add up the results. For a submanifold-with-boundary, the first variation is given (decomposing the vector field into normal and tangential components and using the divergence theorem for the latter, see \cite{Maggi} or \cite{SimonNotes}) by the sum of two terms: one is the integration (on the interior) of the vector field dotted with minus the mean curvature vector; the second is given by minus the integration on the boundary of the vector field dotted with the inward conormal at the boundary. We obtain that the first variation of $|\p^* F_k|$ is represented by the vector-valued Radon measure whose abosolutely continuous part is $-(g-\lambda) \mathcal{H}^2 \res (\Gamma_j \setminus T) \nu$ (where $\nu$ is the outer unit normal to $F_k$) and whose singular part is $-2({\Hc}^{1}\res T)\hat{n}$, where $\hat{n}$ is the conormal to either of the two surfaces-with-boundary at $T$, oriented towards the interior of (either) hypersurface-with-boundary. (Note that $\hat{n}$ spans, with the tangent to $T$, the tangent plane to $D_j$ along $T$, for $j\in \{1,2\}$.)

With this knowledge of the first variation, we can exhibit an ambient deformation of $F_k$ that decreases $\text{Per}_U(\cdot) + \int_{\cdot \cap U} g-\lambda |\cdot \cap U|$ to first order (for $k$ fixed). Let $\hat{n}$ be a smooth extension to $U$ of the the unit vector orthogonal to $T$, tangent to $D_j$ and pointing to the interior of $F_k$ (the above conormal). Such an extension is possible, by taking the initial $U$ smaller if necessary. Let $\rho\in C^1_c(U)$ with $\rho\geq 0$ and $\rho\equiv 1$ on $B$, where $B$ is an open ball centred at $0$. Define $X\in C^1_c(U;\R^3)$ by $X=\rho \hat{n}$. Let $\nu$ be the outer normal to $F_k$ on $\Gamma_k \setminus T$. Then the first variation with respect to $\text{Per}_U(\cdot) + \int_{\cdot \cap U} g-\lambda |\cdot \cap U|$ evaluated on $X$ is given by (using the first variation formula \cite[Proposition 17.8]{Maggi} for the potential and volume terms)
$$\int_{\Gamma_k\setminus T} (\lambda-g)\nu  \cdot X - 2\int_T \hat{n} \cdot X +\int_{\Gamma_k\setminus T} g \nu \cdot X - \int_{\Gamma_k\setminus T} \lambda \nu \cdot X =- 2\int_T \hat{n} \cdot X <0$$ by the choice of $X$ (with integration with respect to ${\Hc}^{1}$ on $T$ and ${\Hc}^{2}$ on $\Gamma_k\setminus T$).

The above implies (through Proposition \ref{Prop:Lagr_mult}) the existence of ambient volume-preserving deformations, separately for $F_1$ and $F_2$, that decrease to first order $\mathcal{E}_U$. Such volume-preserving deformations can be built with initial velocity still equal to $X$ in a neighbourhood of $T$. For that, it suffices (for each $k$) to combine the previous deformation of $F_k$ with another ambient deformation of $F_k$ that only acts inside a compact subset $U\setminus (T\cup B)$ and that balances the volume-constraint (this corresponds to a choice of $Y\in C^1_c(U\setminus (T\cup B))$ in step 3 of the proof of Proposition \ref{Prop:Lagr_mult}).

Denote the deformations just obtained by $F_k^t$ for $t\in[0,\eps)$, $F_k^0=F_k$. By construction $\overline{F}_1^t \cap \overline{F}_2^t\subset T$,
The union $F_1^t \cup F_2^t$ gives a one-parameter \textit{non-ambient} volume-preserving deformation of $F_1 \cup F_2$ with $\mathcal{E}_U$ decreasing to first order at $t=0^+$; moreover, this is a deformation of break up type, as in Definition \ref{Dfi:coalescence_a} (bottom row of Figure \ref{fig:touching_sing_deform}). (Note also that the initial speed for this deformation is bounded.) We have produced a deformation that contradicts the stationarity assumption of Definition \ref{Dfi:station_stable_one_sided}.

\begin{oss}
In case 1, the sets $F_1$ and $F_2$ that we introduced above satisfy $F_1\cup F_2 = U\setminus \overline{E}$ and the exact same argument leads to a volume-preserving deformation of coalescence type that contradicts stationarity (top row of Figure \ref{fig:touching_sing_deform}).
\end{oss}

\medskip

\begin{figure}[h]
\centering
 \includegraphics[width=6cm]{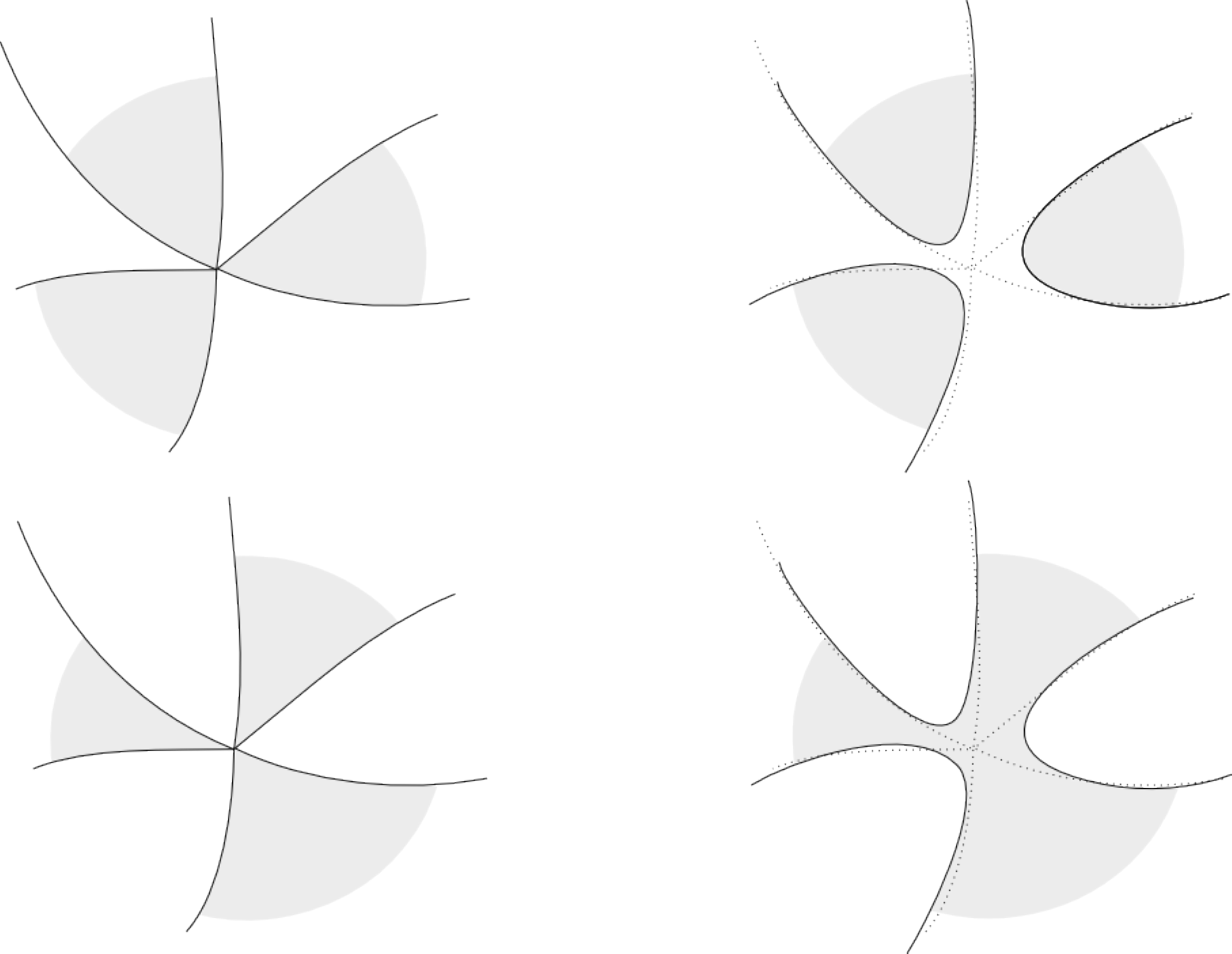} 
\caption{Top row: virtual deformation ``breakup''. Bottom row: virtual deformation ``coalescence''. The boundary is deformed in the same way in both cases. The starting configuration is \textbf{(b)} (one should consider cylinders over the depicted sets).}
 \label{fig:class_sing_deform}
\end{figure}

Passing to case \textbf{(b)}, stationarity for ambient deformations (and the fact that $\overline{E} \cap U$ is connected) implies that there exists $\lambda \in \R$ such that mean curvature is $(g-\lambda)\nu$ on $(\p E \setminus T) \cap U$, where $\nu$ is the outer unit normal. Upon redefining $g$ and $\lambda$ (by a change of sign) if needed, we assume that each wedge of $E\cap U$ has an angle smaller than $\pi$ at $0$ (in the slice orthogonal to $T$ at $0$), where by wedge we mean each connected componenent of the interior of $E\cap U$. This is possible because an angle $\geq \pi$ cannot be present simultaneously for $E$ and for $U\setminus E$. We choose the indexation of $L_k$ in the decomposition given in \textbf{(b)} so that each wedge is bounded by $L_k \cup L_{k+1}$ with $k$ odd.

Then for each fixed odd $k$, the (open) wedge $W$ bounded by $L_k \cup L_{k+1}$ admits a one-sided ambient deformation $\{W^t\}_{t\in[0,\eps)}$ that is volume-preserving, has bounded initial speed, and that decreases to first order the energy $\text{Per}_U(\cdot) +\int_{\cdot \cap U} g$. The initial speed can be chosen to point towards the interior of $W$ along $T$, therefore we can ensure that the closures of any two (distinct) deformed wedges only intersect at points of $T$. All of this follows by arguments very similar to those used to treat \textbf{(a)}, which essentially come down to the observation that the leading term in the first variation for each single wedge is the boundary contribution at $T$ (of each surface-with-boundary $L_j$), and it points to the exterior of the wedge thanks to the condition that the opening angle is $<\pi$. (So that pushing a wedge to its interior in a neighbourhood of $0\in T$ will decrease the energy.)
We then produce, by taking the union of the ambient deformations for each wedge, a non-ambient volume-preserving deformation of $E\cap U$, with bounded initial speed, of break up type (as in Definition \ref{Dfi:coalescence_b}) that decreases $\mathcal{E}_U$ to first order (top row of Figure \ref{fig:class_sing_deform}). If we had passed to the complement of $E$ in the beginning, then this is a coalescence type deformation for $E$ (bottom row of Figure \ref{fig:class_sing_deform}).\footnote{If the acute angle condition is valid for all wedges of $E$ and for all wedges of $U\setminus E$, then our proof, repeated for $E$ and for $U\setminus E$, gives both a coalescence and a break up deformation of $E$ that reduce energy. If the acute angle condition is valid for $E$ but not for $U\setminus E$, then our proof produces only a break up deformation of $E$, while if the acute angle condition is valid for $U\setminus E$ but not for $E$ then our proof produces only a coalescence deformation. In fact, even in the latter two cases, it is possible to decrease energy both by a coalescence and by a break up deformation, however the proof given needs to be adapted: one has to consider the case in which a (single) wedge has opening angle larger than $\pi$ --- its contribution to the first variation would be dominated by the contribution coming from another wedge that has an acute angle.}

\medskip

In conclusion, we have proved that \textbf{(a)} and \textbf{(b)} cannot be present in $\Om$ under the assumptions of Theorem \ref{thm:main}, specifically under the stationarity condition of Definition \ref{Dfi:station_stable_one_sided}.

\section{Ruling out (a') by second variation}
\label{finalproof}

In Section \ref{C} we have proved that under the assumptions of Theorem \ref{thm:main} the configurations \textbf{(a)} and \textbf{(b)} are nowhere present in $\Om$. This, together with Proposition \ref{Prop:Lagr_mult_mean_curv}, ensures that the assumptions of Theorem \ref{thm:main} imply those of Theorem \ref{1} and lets us conclude that $\p E \cap \Om$ is smoothly embedded except possibly at isolated points, around which the only possible structure is \textbf{(a')}. To complete the proof of Theorem \ref{thm:main} we only need to establish Proposition \ref{Prop:a'}, that is we prove that configuration \textbf{(a')} violates the equilibrium assumptions of Theorem \ref{thm:main}. More precisely we will show that it violates the stability condition of Definition \ref{Dfi:station_stable_one_sided}.

\medskip

Assume that in the open set $U\subset \subset \Om$ we have that $\p E \cap U$ is as in \textbf{(a')}. Note that both $\overline{E} \cap U$ and $\overline{U\setminus E}$ are connected (regardless of which is $E\cap U$ and which its complement).
By choosing coordinates suitably we fix $U$ to be the cylinder $B^2_{3R}(0) \times (-R,R)\subset \R^2\times \R$, for some positive $R$, and $\p E \cap U= \text{graph}(u) \cup \text{graph}(v)$ with $u\leq v$, $u, v \in C^2(B_{3R})$, $u(0)=v(0)=0$, $Du(0) = Dv(0) =0$, $u<v$ on $B^2_{3R}(0)  \setminus \{0\}$.  We will use cartesian coordinates $(x,y,z)$ and cylindrical coordinates $(r,\theta, z)$. To lighten notation, we will also write $B^2_s$ for $B^2_s(0)$. There exists $\lambda \in \R$ for which the second stationarity condition in Proposition \ref{Prop:Lagr_mult} holds; as usual, we can make the choice that $E\cap U$ is given by the regions below $u$ and above $v$, upon changing sign to $g$ and $\lambda$ if needed. The graphs of $u,v$ then have mean curvatures respectively $(g-\lambda)$ and $-(g-\lambda)$ with respect to the upward pointing normal.\footnote{ The condition on the mean curvature is verfied thanks to Proposition \ref{Prop:Lagr_mult_mean_curv} at any point $\neq 0$, since for a smooth surface the classical mean curvature agrees with the generalised mean curvature in the varifold sense; by the smoothness of $u,v$ and $g$ the conclusion extends across $0$ for each graph.} \footnote{By the maximum principle it must hold $\lambda>g(0)$ for this configuration to arise, however we will not make use of this fact.} By choosing $R$ sufficiently small we ensure that $\sup_{B^2_{s}}|u|, \sup_{B^2_{s}}|v| < s$ and that the areas of $\text{graph}(u) \cap (B^2_{s} \times (-s,s))$ and $\text{graph}(v) \cap (B^2_{s} \times (-s,s))$ are bounded by $2\pi s^2$ for all $s\leq R$. 

\medskip

Denote by $\p_r$ the vector field $\frac{(x,y,0)}{\sqrt{x^2+y^2}}$ on $U \setminus \{x=y=0\}$. Let $\chi$ be a smooth function on $B^2_{3R}(0)\times (-R,R)$, independent of the variable $z$, with $0\leq \chi \leq 1$, and that is identically $1$ on $r \leq R_0$ and identically $0$ on $r\geq 2R_0$, with $R_0\in (0,R/6)$ to be chosen. We will consider the vector field $\chi \p_r$ on $U \setminus \{x=y=0\}$. Let $\Phi_t$, for $t\in (0,R_0/2)$, denote the flow associated to $\chi \p_r$. This means that $\Phi_t(x)=\gamma(t)$, where $\gamma$ solves $\gamma'(t)=\chi(\gamma(t)) \p_r$, $\gamma(0)=x \in U \setminus \{x=y=0\}$ (that is, $\gamma$ is the flow line of $\chi \p_r$ starting at $x$). Then $\{\Phi_t\}_{t\in[0,R_0/2)}$ is a one-sided one-parameter family of diffeomorphisms, with $\Phi_t$ mapping $(B^2_{3R} \setminus \{0\}) \times (-R,R)$ onto $(B^2_{3R} \setminus B_t^{2}) \times (-R,R)$. Also note that $\Phi_t \circ \Phi_s = \Phi_{t+s}$, since the vector field is time independent. Moreover, $\Phi_0=Id$ and, for every $t\in (0,R_0/2)$, $\Phi_t=Id$ on the cylinder $\{r\geq R\}$. This family of diffeomorphisms induces a one-sided deformation $\tilde{M}_t$ of $M_0=\text{graph}(u) \setminus \{0\}$, namely $\tilde{M}_t=\Phi_t(M_0)$. Note that the restriction of $\tilde{M}_t$ to the cylinder $\{r\in (t,  t+R_0/2]\}$ is the graph of $u(r-t,\theta)$ (in cylindrical coordinates). Roughly speaking, the deformation opens up a hole at $0$ and pushes $\text{graph}(u)\setminus \{0\}$ radially in the horizontal $(x,y)$-directions, with no vertical displacement.
We will consider $\tilde{M}_t$ as an embedded surface in $U$. Then $\tilde{M}_t$ does not intersect the cylinder $\{r <t\}$ and the closure $\overline{\tilde{M}}_t$ is a surface-with-boundary in $U$ and the boundary $\p \overline{\tilde{M}_t}$ in $U$ is the circle $\{r=t, z=0\}$. 

The idea is to combine the one-sided deformation $\tilde{M}_t$ of $\text{graph}(u) \setminus \{0\}$ with an analogous deformation $\tilde{N}_t = \Phi_t(N_0)$ of $N_0=\text{graph}(v) \setminus \{0\}$; since $\p \tilde{N}_t=\p \tilde{M}_t$ this gives rise to a coalescence type deformation $\overline{\tilde{M}}_t\cup \overline{\tilde{N}}_t$ of $\p E \cap U$ in $U$. (Note that $\tilde{M}_t \cap \tilde{N}_t = \emptyset$ for all $t$.) We denote by $\tilde{E}_t$ the Caccioppoli set that agrees with $E$ in the complement of $U$ and is given, in $U$, by 
$$\tilde{E}_t \cap U=\Phi_t\left((E\cap U)\setminus \{(0,0,z):|z|<R\}\right) \cup (\overline{B}^2_t(0)\times (-R,R)),$$
noting that the boundary of $\tilde{E}_t$ in $U$ is given by $\overline{\tilde{M}}_t\cup \overline{\tilde{N}}_t$. (To see that $\tilde{E}_t$ is well-defined, recall that $\{r<t\}\times \{|z|>t\}$ is contained in $E$ for all $t\leq R_0/2$, by the initial choice of $U$, and note that $\tilde{E}_t = E$ in a neighbourhood of $\p U$ for all $t\leq R_0/2$.) In order to ensure the volume-preserving condition we will modify this deformation in the cylinder $\{2R<r<3R\}\times (-R,R)$, in which $\tilde{E}_t=E$ for all $t$.

For this purpose, we analyse the function $\mathbb{V}(t)=|\tilde{E}_t\cap U|$. The first variation of volume at $t=0^+$, i.e.~$\mathbb{V}'(0)$, is not affected by the volume of the cylinder $\{r\leq t\}\times (-R, R)$ (which we included in $\tilde{E}_t$), since this volume is $2\pi R t^2$, so its rate of change at $t=0^+$ is $0$. Hence $\mathbb{V}'(0)$ is the same as the first variation of volume for $\Phi_t\left((E \cap U)\setminus \{(0,0,z):|z|<R\}\right)$ at $t=0^+$. Recall that $\Phi_t$ is a diffeomorphism of $(B^2_{3R}\setminus \{0\}) \times (-R,R)$ onto $(B^2_{3R}\setminus B^2_t) \times (-R,R)$. Then the first variation of volume can be computed as in the proof of \cite[Proposition 17.8]{Maggi} and is equal to $\int_{E\cap (U\setminus \{(0,0,z):|z|<R\})}\text{div}\,(\chi \p_r)$. This divergence is bounded by $1/r$ in a neighbourhood of $\{(0,0)\} \times (-R, R)$, and is thus summable on the given domain. Approximating this integral by $\int_{E\cap (U\setminus \{(r,z):|z|<R, r<\rho\})}\text{div}\,(\chi \p_r)$ with $\rho>0$, $\rho \to 0$, we can use the divergence theorem to obtain that the first variation of $\mathbb{V}$ is given by 
$$\mathbb{V}'(0)=\lim_{\rho \to 0} \int_{\p^* E \cap (U\setminus \{(r,z):|z|<R, r<\rho\})} \nu \cdot (\chi \p_r) + \lim_{\rho \to 0} \int_{\p\{(r,z):|z|<R, r<\rho\})} (-\p_r) \cdot (\chi \p_r).$$
The $(-\p_r)$ in the second integral is the unit outer normal to the surface of integration, and $\nu$ in the first integral is the unit outer normal to $E$. The second integrand is bounded in absolute value by $4\pi\rho R$, hence the first variation of volume is $\mathbb{V}'(0) =\int_{\text{graph}(u)} \chi\p_r \cdot \nu+\int_{\text{graph}(v)} \chi \p_r \cdot \nu$. We note from this expression that $\mathbb{V}'(0)$ tends to $0$ as $R_0\to 0$.

We will also need to compute $\mathbb{V}''(0)$. To that end, we first show that $\mathbb{V}'(t)= \int_{M_t \cup N_t}  \chi \nu \cdot \p_r$ for all $t>0$. This is seen as follows. 
Recall that $\Phi_s$ is a diffeomorphism from $(B_{3R}^2 \setminus B_t^2) \times (-R, R)$ onto $(B_{3R}^2 \setminus B_{t+s}^2) \times (-R, R)$ and that $\Phi_t \circ \Phi_s = \Phi_{s+t}$. Then for any $t>0$ the derivative $\mathbb{V}'(t)$ can be computed as sum of two terms: 
$$\left.\frac{d}{ds}\right|_{s=0^+}\left|\Phi_s\left(E_t \cap  \left((B_{3R}^2 \setminus B_{t+s}^2) \times (-R, R)\right)\right)\right| + \left.\frac{d}{ds}\right|_{s=0^+}|B_{t+s}^2 \times (-R, R)|.$$
The first term (using \cite[Proposition 17.8]{Maggi}, with the funcion $g$ therein replaced by the constant $1$), is equal to $\int_{E_t \cap  \left((B_{3R}^2 \setminus B_{t}^2) \times (-R, R)\right)} \text{div} (\chi \p_r)$, for which we use the divergence theorem (leading to the first two terms on the right-hand-side of the next identity); in conclusion,
$$\mathbb{V}'(t) = \int_{\p E_t \cap  \left((B_{3R}^2 \setminus B_{t}^2) \times (-R, R)\right)}  \chi \p_r \cdot \nu + \int_{\p B^2_t \times (-R, R)} \chi \p_r \cdot (-\p_r) + \int_{\p B^2_t \times (-R, R)} 1.$$
(The $(-\p_r)$ in the second integrand is the unit outer normal on the surface of integration.) Since $\chi=1$ on the surface of integration for the second integrand, the second and third integrand cancel each other; note also that $\p E_t \cap  \left((B_{3R}^2 \setminus B_{t}^2) \times (-R, R)\right)= \p E_t \cap U = M_t \cup N_t$. Thus, as claimed (with $\nu$ denoting the outer unit normal to $\tilde{E}_t$)
$$\mathbb{V}'(t) = \int_{M_t}  \chi \p_r \cdot \nu+\int_{N_t}  \chi \p_r \cdot \nu.$$

We now compute $\mathbb{V}''(0)$ by differentiating the expression obtained for $\mathbb{V}'(t)$. Using that both for $M_t$ and $N_t$ we have $\p_r \cdot \nu=0$ on the boundaries $\p M_t$ and $\p N_t$ (so that there is no boundary contribution when differentiating the area measure), we get
$$\mathbb{V}''(0)= \int_{M_0 \cup N_0} \chi^2 (\nu \cdot \p_r) (\vec{H}\cdot \p_r) + \int_{M_0 \cup N_0} \left.\frac{d}{dt}\right|_{t=0} (\chi \nu \cdot \p_r),$$
where $\vec{H}$ stands for the mean curvature vector of $M_0 \cup N_0$. Hence
$$|\mathbb{V}''(0)|\leq  \int_{M_0 \cup N_0}\chi^2|\vec{H}| + C\int_{M_0 \cup N_0} \chi|D(\chi \p_r)|+ \int_{M_0 \cup N_0} \left|\left.\frac{d \nu}{dt}\right|_{t=0^+}\right| ,$$
for some (dimensional) constant $C$. We have (by direct computation) $\left|\left.\frac{d \nu}{dt}\right|_{t=0^+}\right|\leq CD(\chi \p_r)$. Recall now that $\vec{H}$ is bounded, $|D\p_r|\leq C/r$ is summable, the area of $M_0 \cup N_0$ is bounded by $2\pi R_0^2$ in the set where $\chi\neq 0$, and $\int_{M_0 \cup N_0} |D \chi|$ tends to $0$ as $R_0 \to 0$ ($D\chi$ is bounded by $\frac{2}{R_0}$ on an annulus of area at most $2\pi R_0^2$ and vanishes otherwise). In particular we see that $\mathbb{V}''(0)$ tends to $0$ as $R_0\to 0$.

\medskip

\medskip

\noindent \textit{Relevant volume-preserving coalescence deformation.} Consider the two subdomains $U_1=B^2_R \times (-R,R)$ and $U_2=(B^2_{3R} \setminus \overline{B}^2_R)\times (-R,R)$. By construction $\tilde{E}_t$ and $E$ coincide in $U_2$. Let $\zeta_0 \in C^1_c(M_0 \cap U_2)$ be such that $\int_{M_0} \zeta_0 \neq 0$ (this choice is independent of $R_0$). Set $\zeta=\frac{-\mathbb{V}'(0)}{\int_{M_0}\zeta_0}\zeta_0$, so that $\int_{M_0} (\zeta \nu) \cdot \nu = -\mathbb{V}'(0)$, and extend $\zeta \nu$ to a vector field in $U_2$, compactly supported away from $N_0 \cap U_2$. Following the argument in \cite[Lemma 2.4]{BarbDoCarmo} (also recalled in step 1 of the proof of Proposition \ref{Prop:Lagr_mult}) we pick another vector field $Y \in C^1_c(U_2 \setminus N_0)$ such that $\int_{M_0} Y\cdot \nu \neq 0$ (the choice of $Y$ is independent of $R_0$) and construct, for $t\in (-t_0,t_0)$ for some $t_0>0$, an ambient deformation $\Psi_t=Id + t\zeta \nu +s(t) Y$ with $s(0)=0$ and $s'(0)=0$, i.e.~with initial speed $\zeta \nu$, and such that $|\Psi_t(E)\cap U_2|=-\mathbb{V}(t)$. This follows from the implicit function theorem, by considering the level set $f(t,s)=0$ for the function $f(t,s)=|(Id+t \zeta \nu + s Y)(E) \cap U_2|+\mathbb{V}(t)$, noting that $f(0,0)=0$ and $\frac{\p f}{\p s}(0,0) = \int_{M_0} Y\cdot \nu \neq 0$. Then the level set is described by $(t,s(t))$ for $t\in (-t_0,t_0)$; moreover, $s'(0)=\frac{-\frac{\p f}{\p t}(0,0)}{\frac{\p f}{\p s}(0,0)}$, which vanishes because $\frac{\p f}{\p t}(0,0) = \mathbb{V}'(0) +\int_{M_0} (\zeta \nu)\cdot \nu = 0$.

We will in fact only need to employ this deformation for $t\in [0,t_0)$. Note that $\Psi_t(M_0)$ remains disjoint from $N_0$ for all $t\leq t_0$. 
Combining the deformation $\Psi_t$, which only acts in $\{2R<r<3R\}\times (-R, R)$, with the deformation $\tilde{M}_t$ introduced above, which only acts in $\{r<R\}\times (-R, R)$, we obtain a deformation $E_t$ of $E$ that is volume preserving in $U$. The boundary of $E$ in $U$ is given by $\overline{M}_t \cup \overline{N}_t$, where $M_t$ is equal to $\tilde{M}_t$ in $\{r<2R\} \times (-R, R)$ and equal to $\Psi_t(M_0)$ in $\{2R\leq r<3R\} \times (-R, R)$, while $N_t=\tilde{N}_t$. Recalling that $\mathbb{V}'(0) \to 0$ as $R_0\to 0$ ($\zeta_0$ and $\nu$ are fixed) we note that $\sup_{U_2} |\zeta \nu|$ and $\sup_{U_2} |D (\zeta \nu)|$ can be made arbitrarily small by choosing $R_0$ sufficiently small.

The deformation $\{E_t\}_{t\in[0,t_0)}$ is of coalescence type: $\p E_t \cap U$ is the image of a $C^1$ map $\psi_t:S^1 \times (-1,1) \to U$, smoothly depending on $t$, whose differential is allowed to vanish (for each $t$) on $S^1 \times \{0\}$ and that is a smooth embedding away from the image of $S^1 \times \{0\}$ (as required in Definition \ref{Dfi:coalescence_a'}). At $t=0$ the circle $S^1 \times \{0\}$ is mapped to $0$, while for $t>0$ it is mapped to $\p M_t=\p N_t$. 
The initial speed of the deformation just exhibited is given by $X=\chi \p_r + \zeta \nu$. This is a well-defined $\R^3$-valued map on $S^1 \times (-1,1)$ (and bounded as in Definition \ref{Dfi:station_stable_one_sided}). We remark that on $S^1 \times \{0\}$ it fails to be identifiable with an ambient vector field in $U$; at $(p,0)\in S^1 \times (-1,1)$ the vector field is given by $\frac{p}{|p|}$, thinking of $p\in S^1 \subset \R^2 \times \{0\} \subset \R^3$. We stress that the speed of the deformation $\p E_t \cap U_1$ is equal to $\chi \p_r$ (as a vector field on $S^1 \times (-1,1)$), and thus independent of $t\in[0,t_0)$; the speed of $\p E_t \cap U_2$ is equal to $\zeta \nu + s(t) Y$m hence $t$-dependent (the latter is also well-defined as a vector field on $U_2$). 

The volume-preserving condition implies, arguing as we did for $\mathbb{V}'(0)$ (using the approximation argument with $\rho\to 0$), that 
$$\int_{\text{graph}(u)} X \cdot \nu+\int_{\text{graph}(v)} X \cdot \nu =0.$$

In view of forthcoming arguments, we compute also $s''(0)$. Since $\frac{\p f}{\p t}=0$ at $(t,s)=(0,0)$, the implicit function theorem gives $s''(0)=\frac{-\frac{\p^2 f}{\p t^2}}{\frac{\p f}{\p s}}$, with the derivatives on the right-hand-side evaluated at $(t,s)=(0,0)$. To compute $\frac{\p^2 f}{\p t^2}$ we recall \cite[(17.11)]{Maggi} and argue as in the proof of \cite[Proposition 17.8]{Maggi}, keeping also the terms that are quadratic in $t$. We obtain $\left.\frac{\p^2 f}{\p t^2}\right|_{(t,s)=(0,0)} = \mathbb{V}''(0)+\int_{E\cap U_2} (\text{div}(\zeta \nu))^2 -\int_{E\cap U_2} (\text{trace}(D(\zeta \nu))^2)$. All three terms tend to $0$ as $R_0\to 0$: for $\mathbb{V}''(0)$ this was shown earlier, for the second and third terms this follows from the fact (see above) that $\sup_{U_2}|\zeta \nu|$ and $\sup_{U_2} |D(\zeta \nu)|$ tend to $0$ as $R_0\to 0$. The denominator $\frac{\p f}{\p s}$ is determined by the choice of $Y$ (so independent of $R_0$). In conclusion, $s''(0) \to 0$ as $R_0 \to 0$.

\medskip

\medskip

\noindent \textit{First variation.} We will now check that the first variation of $\mathcal{E}_U(E_t)=\text{Per}_U(\p^* E_t) + \int_{E_t\cap U} g$ at $t=0^+$ (along the deformation $\{E_t\}_{t\in[0,t_0)}$ just constructed) is
\begin{equation}
\label{eq:first_var_0}
\left.\frac{d}{dt}\right|_{t=0^+} \mathcal{E}_{U}(E_t)=-\int_{\text{graph}(u)} \vec{H} \cdot X   + \int_{\text{graph}(u)} g\nu \cdot X-\int_{\text{graph}(v)} \vec{H} \cdot X   + \int_{\text{graph}(v)} g\nu \cdot X,
\end{equation}
where $\vec{H}$ stands for the mean curvature vector (of $\text{graph}(u)$ and $\text{graph}(v)$ respectively, in the two corresponding integrals). 
It is convenient to compute the above first variation separately in the two subdomains $U_1=B^2_R\times (-R,R)$ and $U_2=(B^2_{3R}\setminus \overline{B}^2_R)\times (-R,R)$. 
The first variation in $U_2$ is given by the well-known formulae (see \cite{BarbDoCarmo}, \cite{Maggi}), since in this domain the deformation is an ambient one, induced by a one-parameter family of diffeomorphisms whose initial speed is the vector field  $\zeta \nu \in C^1_c(U_2)$. We thus obtain 

$$\left.\frac{d}{dt}\right|_{t=0} \mathcal{E}_{U_2}(E_t)=-\int_{\text{graph}(u)} \zeta \vec{H} \cdot  \nu   + \int_{\text{graph}(u)} g \zeta \nu \cdot \nu -\int_{\text{graph}(v)} \zeta \vec{H} \cdot \nu    + \int_{\text{graph}(v)} g \zeta \nu \cdot \nu.$$

In $U_1$ the initial speed $\chi\p_r$ is not compactly supported in $U_1$, and the deformation moves the boundary of $\overline{M_t}$ and $\overline{N_t}$. For future purposes, for $U_1$ it will be convenient to compute the first variation for an arbitrary $t$. Recalling that $E_t \cap U_1 = \Phi_t(E \cap U_1) \cup (\overline{B}^2_t \times (-R,R))$, and using that $\Phi_s \circ \Phi_t = \Phi_{s+t}$, the first variation at $t\in[0,t_0)$ in $U_1$ is given by (as we will justify below)
\begin{equation}
\label{eq:first_var_t}
\frac{d}{dt} \mathcal{E}_{U_1}(E_t) = -\int_{M_t} \vec{H} \cdot (\chi \p_r) + \int_{\p M_t} \vec{n} \cdot (\chi \p_r) + \int_{M_t} g\nu \cdot (\chi \p_r)  
\end{equation}
$$ -\int_{N_t} \vec{H} \cdot (\chi \p_r) + \int_{\p N_t} \vec{n} \cdot (\chi \p_r) + \int_{N_t} g\nu \cdot (\chi \p_r), $$
where $\vec{H}$ stands for the mean curvature vector (of $M_t$ and $N_t$ respectively, in the two corresponding integrals) and $\vec{n}$ is the unit conormal at $\p M_t$, pointing away from $M_t$, or the unit conormal at $\p N_t$, pointing away from $N_t$. (By construction, the unit conormal is the same for $M_t$ and $N_t$; recall that $\p M_t = \p N_t$.)

The first, second, fourth and fifth terms in (\ref{eq:first_var_t}) arise from the first variation of area of a surface-with-boundary (see \cite{SimonNotes}).
The third and sixth terms of (\ref{eq:first_var_t}) arise from the first variation of $\int_{E_t\cap U_1} g$. For $t>0$ the latter claim follows by splitting the domain $E_t\cap U_1$ into the two subdomains $(E_t \cap U_1) \setminus (B_t \times (-R, R))$ and $B_t \times (-R, R)$. The contribution coming from the first subdomain is computed as in the proof of \cite[Proposition 17.8]{Maggi}, using that $\Phi_s$ is a diffeomorphism from $(B_{3R}\setminus B_t) \times (-R, R)$ onto $(B_{3R}\setminus B_{t+s}) \times (-R, R)$, and gives $\int_{E\cap ((B_{3R}\setminus B_t) \times (-R, R))} \text{div}\,(g \chi \p_r)$. Using the divergence theorem we obtain the third and sixth terms of (\ref{eq:first_var_t}) and an additional summand $\int_{\p B^2_t \times (-R, R)}g \chi \p_r\cdot (-\p_r)$. The contribution from the second subdomain gives rise to $\int_{\p B^2_t \times (-R, R)}g$. These two last integrals on the surface $\p B^2_t \times (-R, R)$ cancel each other (since $\chi=1$ on the surface of integration). In the case $t=0$ we argue as we did earlier for $\mathbb{V}'(0)$ (using the approximation argument with $\rho\to 0$) to reach the same conclusion.\footnote{For $t_0>0$ one can give an alternative argument as follows: the deformation of $E_{t_0}$ that we are considering can be equivalently induced by the one-parameter family of diffeomorphisms $Id+s\tilde{X}_s$, where $\tilde{X}_s=\psi (\chi \p_r)$ and $\psi$ is a smooth function that only depends only on $r$ and is equal to $1$ for $r\geq t_0/2$ and equal to $0$ for $r\leq t_0/4$. This is well-defined as a vector field in $U$, so the usual formula for the first variation of volume applies.}

Recalling that $\p M_t$ is the circle in the $(x,y)$-plane with centre at $0$ and radius $t$ (with $t\leq R_0/2$) we have $\hat{n}=-\p_r=-\chi \p_r$ on $\p M_t$, for all $t\in [0, t_0)$. This implies that the second and fifth terms of (\ref{eq:first_var_t}) are each equal to $-2\pi t$. 

Not surprisingly, sending $t\to 0^+$ in (\ref{eq:first_var_t}) we find $\left.\frac{d}{dt}\right|_{t=0^+} \mathcal{E}_{U_1}(E_t)$, which combined with $\left.\frac{d}{dt}\right|_{t=0} \mathcal{E}_{U_2}(E_t)$ gives (\ref{eq:first_var_0}) again. Substituting in (\ref{eq:first_var_0}) $\vec{H} = (g-\lambda)\nu$ for a certain $\lambda \in \R$ (as found in the beginning of this section), and recalling $\int_{\text{graph}(u)} X\cdot \nu + \int_{\text{graph}(v)} X\cdot \nu =0$, we get that the first variation at $t=0^+$ vanishes:
$$\left.\frac{d}{dt}\right|_{t=0^+} \mathcal{E}_{U}(E_t)=0.$$

\medskip

\medskip

\noindent \textit{Second variation.} We will now show that the second variation of $\mathcal{E}_U(E_t)=\text{Per}_U(\p^* E_t) + \int_{E_t\cap U} g$ at $t=0^+$ (for the same deformation $E_t$) is strictly negative, contrary to the assumption (we just checked that we have vanishing one-sided first variation, so the implication in Definition \ref{Dfi:station_stable_one_sided} yields that that the second variation along this deformation has to be non-negative). 

Again, the computation is conveniently carried out separately in the two subdomains $U_1, U_2$. For $U_2$ we rewrite $\mathcal{E}_{U_2}(C)=\text{Per}_{U_2}(C)+\int_{C\cap U_2}(g-\lambda) + \lambda |C\cap U_2|$ (for the $\lambda$ identified in the beginning of the section). The well-known formula (\cite{BarbDoCarmo}, \cite{BW1}) for the second variation of $\text{Per}_{U_2}(C)+\int_{C\cap U_2}(g-\lambda)$ for the set $E$ (that is stationary with respect to this functional in $U_2$, i.e.~$\p E$ has mean curvature $(g-\lambda)\nu$) yields 
$$\left.\frac{d^2}{dt^2}\right|_{t=0}\mathcal{E}_{U_2}(E_t)=\left(\int_{\p E \cap U_2} (D_{\nu} g - |A_{\p E}|^2) \zeta^2 + |\nabla \zeta|^2 \right)+ \lambda \left.\frac{d^2}{dt^2}\right|_{t=0} |E_t\cap U_2|,$$
where $D_\nu$ is the derivative in the direction $\nu$ and $\nabla$ is the gradient on the surface.
For the second term on the right-hand-side, we follow the proof of \cite[Proposition 17.8]{Maggi} using the second (rather than first) order expansion in $t$ of the Jacobian of the deformation. Expanding the diffeomorphism as $Id+t (\zeta \nu) + \frac{t^2}{2} Z +o(t^2)$, where $Z=s''(0)Y$ by construction, and using $\text{det}(I + t A + \frac{t^2}{2}B)=1+t\,\text{trace}(A) + \frac{t^2}{2}(\text{trace}(B) + (\text{trace}(A))^2-\text{trace}(A^2))+o(t^2)$, we get 
$$\left.\frac{d^2}{dt^2}\right|_{t=0} |E_t\cap U_2|=\int_{E\cap U_2} \text{div}Z +\int_{E\cap U_2}(\text{div}(\zeta \nu) )^2 - \int_{E\cap U_2}\text{trace}((D(\zeta \nu) )^2)=$$
$$=\int_{\p E\cap U_2} Z \cdot \nu +\int_{E\cap U_2}(\text{div}(\zeta \nu) )^2 - \int_{E\cap U_2}\text{trace}((D(\zeta \nu) )^2).$$
The last two integrals tend to $0$ as $R_0\to 0$, since so does $\sup_{U_2}|D (\zeta \nu)|$. For the first integral, recall that $s''(0)\to 0$ as $R_0 \to 0$, while $Y$ is fixed independently of $R_0$, so this term also tends to $0$ as $R_0\to 0$ and we conclude that
$$\left.\frac{d^2}{dt^2}\right|_{t=0} |E_t\cap U_2| \to 0 \,\,\text{ as } R_0\to 0.$$

For $U_1$, instead, we compute the derivative at $t=0^+$ of (\ref{eq:first_var_t}). The derivatives of the terms $\int_{\p M_t} \vec{n} \cdot (\chi \p_r)$ and $\int_{\p N_t} \vec{n} \cdot (\chi \p_r)$ at $t=0^+$ are each equal to $-2\pi$. This will be true regardless of how we choose $R_0$. 

The term $-\int_{M_t} \vec{H} \cdot (\chi \p_r)$ has derivative at $t=0^+$ given by 
\begin{equation}
 \label{eq:termHX}
\int_{M_0} (\vec{H}\cdot (\chi \p_r))^2 - \int_{M_0}\left. \frac{\p \vec{H}}{\p t}\right|_{t=0^+} \cdot (\chi \p_r) - \int_{M_0} \vec{H}\cdot \left( D(\chi \p_r) \cdot (\chi \p_r)  \right).
\end{equation}
(There is no boundary contribution since $\vec{H}$ and $(\chi \p_r)$ are bounded on $\overline{M_0}$ and $\p M_0$ is reduced to a point; in fact, one also has that $\vec{H}\cdot (\chi \p_r)$ tends to $0$ as we approach $\p M_0$.) For the second summand $\int_{M_0} \left. \frac{\p \vec{H}}{\p t}\right|_{t=0^+} \cdot (\chi \p_r)$ of (\ref{eq:termHX}) we make use of the computations in \cite[Appendix]{BarbDoCarmo}, noting that those computations are local and can be repeated\footnote{In our case, unlike in \cite{BarbDoCarmo}, the deformation of $M_0$ is not performed at fixed boundary, but $M_t$ moves so that the speed is orthogonal to the normal to $M_t$ at $\p M_t$. The contribution of the moving boundary to the first and second variations of area has been isolated in the term $\int_{\p M_t} \vec{n}\cdot (\chi \p_r)$. The remaining terms are treated as in \cite{BarbDoCarmo}, since those computations are local and no not use the fixed boundary condition until after formula \cite[p. 353]{BarbDoCarmo}. We point out that $f$ in \cite{BarbDoCarmo} corresponds to $\nu \cdot (\chi \p_r)$ in our case, and $\xi$ in \cite{BarbDoCarmo} corresponds to $\chi \p_r$. Also note that the convention in \cite{BarbDoCarmo} is to have the mean curvature $H_0$ equal to the average of the principal curvatures (while we take the sum).} until formula \cite[p. 353]{BarbDoCarmo}. The only difference is given by the fact that, instead of the constant $n H_0$, the third summand on the right-hand-side of the formula \cite[p. 353]{BarbDoCarmo} will have (in the integral) the mean curvature function $(\lambda-g)$. We thus find

$$\int_{M_0} \left.\frac{\p \vec{H}}{\p t}\right|_{t=0^+} \cdot (\chi \p_r) = -\int_{M_0}  ((\chi \p_r)\cdot \nu) \left(\Delta_{M_0} ((\chi \p_r)\cdot \nu)\right) - \int_{M_0}|A_{M_0}|^2 ((\chi \p_r)\cdot \nu)^2 +$$ $$+\int_{M_0} \nabla ((\chi \p_r)\cdot \nu) \,\,G,$$
where $\nabla$ is the gradient of the surface. The (vector-valued) function $G$ has modulus bounded above by a constant that only depends on $\sup_{M_0} |A_{M_0}|$, $\sup_U |D g|$, $\|u\|_{C^1(B^2_{3R})}$. (This comes from crude bounds on the last three terms of the analogue of formula \cite[p. 353]{BarbDoCarmo}; this will suffice for our purposes.)
Recall the following facts. The second fundamental form $A_{M_0}$ (and similarly the mean curvature $\vec{H}$) is bounded in $U$, since $M_0\subset D_1$ and $D_1$ is a smooth graph;
the area of $M_0$ grows at a quadratic rate around $0$. We then obtain the following bounds. Firstly,
$$\int_{M_0} ((\chi \p_r)\cdot \nu)^2\leq \int_{M_0} \chi^2 \,\,\,\text{ and } \,\,\, \int_{M_0} \chi$$
can be made as small as we wish by taking $R_0$ sufficiently small. Moreover (using that $\nu$ is normal to the surface), 
$$|\nabla ((\chi \p_r)\cdot \nu)| \leq|A_{M_0}| \chi.$$
Then 
$$\int_{M_0} \left.\frac{\p \vec{H}}{\p t}\right|_{t=0^+} \cdot (\chi \p_r) = \int_{M_0} \left|\nabla_{M_0} ((\chi \p_r)\cdot \nu)\right|^2 - \int_{M_0}|A_{M_0}|^2 ((\chi \p_r)\cdot \nu)^2 +\int_{M_0} \nabla ((\chi \p_r)\cdot \nu) \,\,G\,\,,$$
$$\text{and} \,\,\ \int_{M_0}(\vec{H}\cdot (\chi \p_r))^2= \int_{M_0} |\vec{H}|^2((\chi \p_r)\cdot \nu)^2$$
can be made as small as we wish by choosing $R_0$ suitably small. For the remaining term in (\ref{eq:termHX}), namely
$$\int_{M_0} \vec{H}\cdot \left( D(\chi \p_r) \cdot (\chi \p_r)\right)=\int_{M_0} (\vec{H}\cdot D \chi)\chi + \int_{M_0} \chi^2 \vec{H}\cdot \left(D(\p_r) \cdot \p_r\right),$$
we observe that it can be made as small as we wish by taking $R_0$ sufficiently small, because $|D(\p_r)|\leq C/r$ is summable on $M_0$ and $\int_{M_0} |D\chi|\leq C R_0$ (where $C$ denotes a dimensional constant).
This concludes the proof that we can make the absolute value of the derivative at $t=0^+$ of the term $-\int_{M_t} \vec{H}\cdot X$ as small as we wish, by choosing a small enough $R_0$.

We now consider the derivative of the third term in (\ref{eq:first_var_t}), that is $\int_{M_t} g\,\nu \cdot (\chi \p_r)$, at $t=0^+$. This derivative is given by
$$-\int_{M_0} g\,(\nu \cdot (\chi \p_r)) (\vec{H}\cdot (\chi \p_r)) + \int_{M_0} \left.\frac{d}{dt}\right|_{t=0^+} \left(g \chi \,\nu\cdot \p_r\right) .$$ 
Using again that $\left|\left.\frac{d}{dt}\right|_{t=0^+}\nu\right| \leq C |D(\chi \p_r)|$ (for some dimensional constant $C$) we obtain that this derivative is bounded in modulus by 
$$ C\sup_U(|g|+|\nabla g|)\left( \int_{M_0} \chi |\vec{H}| +\int_{M_0} |\chi | + \int_{M_0} |D \chi|+\int_{M_0} \chi |D(\p_r)|\right).$$
This term can also be made as small as we wish, as above, taking $R_0$ small. 

We repeat identical considerations for $N_0$ to handle the fourth, fifth and sixth terms in (\ref{eq:first_var_t}). 
In conclusion, the second variation with respect to $\mathcal{E}_U$ along $E_t$ at $t=0^+$ is given by $-4\pi$ (the derivative at $t=0^+$ of the second and fifth summands in (\ref{eq:first_var_t})) plus other terms that can be made small in module, say smaller than $\pi$, by a suitable initial choice of $R_0$. Hence we have a strictly negative second variation, contradiction. The structure \textbf{(a')} then cannot be anywhere present in $\Om$. The embeddedness of $\p^* E \cap \Om$ claimed in Theorem \ref{thm:main} is proved. It follows immediately that $\p^* E \cap \Om = \p E \cap \Om$ (when $E$ is its Lebesgue representative) and (using Propositions \ref{Prop:Lagr_mult} and \ref{Prop:Lagr_mult_mean_curv}) that the mean curvature is given by $(\lambda -g)\nu_E$ for $\lambda \in \R$ depending on the connected component of $\overline{E} \cap \Om$. This finishes the proof of Theorem \ref{thm:main}.

\section{Curvature bounds and droplets}
\label{droplets}

An additional question, in fact closely related to the ruling out of configuration \textbf{(a')} in Theorem \ref{thm:main}, is whether it is possible to have regions in which $\p E$ is smoothly embedded, however with points of ``very high curvature''. Ruling out \textbf{(a')} in Section \ref{finalproof} amounts to a negative answer in the ``limit case'' in which a connecting neck between two bulks of liquid is reduced to a point. However, the presence of embedded portions with arbitrarily high curvature is not ruled out in the conclusion of Theorem \ref{thm:main}. We will now give an additional conclusion and some further remarks in this direction.

\medskip

On one hand, we point out that, as we are allowing in Theorem \ref{thm:main} multiple connected components of $\overline{E} \cap \Om$, it is possible, say with $g\equiv 0$, that $E$ is the union of countably many disjoint balls $B_n$ with the radius of $B_n$ tending to $0$ as $n\to \infty$ sufficiently fast. (Note that balls are minimisers of the perimeter for constrained volume, hence stable.) Therefore it is possible to have $|A_{\p E}|$ unbounded in the conclusion of Theorem \ref{thm:main}. 

On the other hand, if we restrict to a connected $E$, or to each single connected component of $E$, we can add the following claim to the conclusions of Theorem \ref{thm:main}:

\medskip

\noindent \textbf{Bounded curvature in the interior.} \textit{Given $W\subset \subset \Om$, for any connected component $E_0$ of $E \cap \Om$ the curvature of $\p E_0 \cap W$ is bounded above, in other words there cannot be a sequence of points $y_n \in \p E_0 \cap W$ with $|A_{\p E}(y_n)|\to \infty$. }

\medskip

\begin{proof}[proof of the claim]
Recall that each connected component $E_0$ of $\overline{E} \cap \Om$ has mean curvature $(\lambda_{E_0} - g) \nu$ for some $\lambda_{E_0} \in \R$ (depending on the connected component). Consider a sequence of points $y_n \in  \p E_0 \cap W$ such that $|A_{\p E}(y_n)|\to \sup_W |A_{\p E}|$. There exist $R>0$ such that $B^3_R(y_n) \subset \Om$ for all $n$, and $K>0$ such that $\mathcal{H}^2(\p E_0 \cap B^3_R(y_n))\leq K$ for all $n$. (The existence of $R$ follows from the fact that $W\subset \subset \Om$. Then, upon possibly making $R$ smaller, we can ensure the second fact for any given $K>1$ thanks to the almost monotonicity formula and to the fact that all points of $\p E \cap \Om$ have density $1$.) Then we can consider the class of Caccioppoli sets $C\subset \R^3$ with embedded boundary in $B_R^3(0)$, such that $\mathcal{H}^2(\p C \cap B_R^3(0)) \leq K$, the mean curvature of $\p C$ in $B_R^3(0)$ is bounded in modulus by $\max_{\overline{W}} |g| + \lambda_{E_0}$, and that are stationary and stable in $B_R^3(0)$ with respect to $\mathcal{H}^2(\p C \cap B_R^3(0)) + \int_{B_R^3(0) \cap C} \tilde{g}$ for volume-preserving ambient deformations, for some smooth $\tilde{g}$ with $|\tilde{g}|\leq \lambda_0 + \max_{\overline{W}} |g|$. For this class the curvature estimates in \cite[Theorem 1]{BCW} (more precisely, the easy generalization to the case of non-constant mean curvature discussed in \cite[Remark 4]{BCW}) give an a priori bound on the curvature. The surfaces $\p E_0 \cap B_R^3(y_n)$, translated by the vector $y_n$, belong to the class considered. Hence $\sup_W |A_{\p E}|<\infty$. (Alternatively, rather than employing \cite[Theorem 1]{BCW}, one can easily adapt the blow up argument in the proof of \cite[Theorem 1]{BCW} to the present situation.)
\end{proof}

\begin{oss}
This boundedness of the curvature, that we obtained from \cite{BCW}, can be interpreted as a reason to expect that \textbf{(a')} must contradict the stability condition in Theorem \ref{thm:main}, with the caveat that coalescence starting from a tiny but embedded meniscus can be induced by an ambient deformation, while the same is not true for \textbf{(a')}. Indeed, coalescence or break up deformations treat \textbf{(a')} as a degenerate immersion of a cylinder and the curvature of this degenerate immersion is infinite.
\end{oss}

\begin{oss}
As a side note, it could be interesting to additionally investigate a priori or quantitative bounds on the curvature, for a suitable class of boundaries. For example, the class of sets $E$ such that $\overline{\p E \cap \Om} \cap \p \Om \neq \emptyset$, with suitable upper bounds on norms of the potential $g$ (and possibly further bounds). 
It appears to be necessary to restrict to an open set $W$ with closure compactly contained in $\Om$. For otherwise we may consider, with $g\equiv 0$ and $\Om = \{z>0\}$, a hemisphere with the equator lying on $\{z=0\}$: as we choose small and smaller radius, the hemisphere will lie outside any fixed $W \subset \subset \Om$ and the curvature will tend to $\infty$. Alternatively, with $g\equiv 0$ and $\Om = \{(x,y,z):|x|< 1, |y|< 1, z>0\}$ we may consider a sequence of surfaces $S_n$, where each surface $S_n$ is a portion of a Delaunay CMC unduloid, with mean curvature independent of $n$, as follows: the rotational axis of the unduloid is the line $\{(0, y, 0)\}$, and the unduloids degenerate towards a union of tangential spheres as $n\to \infty$; the surface $S_n$ is the portion of the unduloid that is in $\Om$. (Recall that the Delaunay family of unduloids depends continuously on a parameter and it interpolates between a cylinder and a union of tangential spheres, keeping the same value for the mean curvature.) Then the maximum of the curvature of $S_n$ in $\Om$ tends to $\infty$ as $n\to \infty$, and this maximum is achieved at a point $x_n\in S_n$ such that $x_n$ accumulate on $\p \Om$ as $n\to \infty$. (These points $x_n$ are on the necks of the unduloids.) Note that both the hemisphere and the surface $S_n$ are graphs over the $(x,y)$-plane, hence stable.
\end{oss}

\medskip

In the small volume regime for $E$ one speaks of ``droplets''. We assume that in Theorem \ref{thm:main} we have no solid supports and $E$ bounded and with connected closure. The analysis in \cite[Section 1.2]{CirMaggi} implies that (from stationarity for volume-preserving ambient deformations with respect to $\mathcal{E}_{\R^3}$) the so-called Alexandrov deficit tends to $0$ when the volume $|E| \to 0$. Therefore for droplets, i.e.~for sufficiently small $|E|$, the Alexandrov deficit is small and, in view of \cite[Theorem 1.1 (i)]{CirMaggi}, the surface $\p E$ has to be given by a perturbation of a collection of spheres. (The $C^2$ embeddedness of $\p E$ is an assumption in \cite{CirMaggi}, while it was a conclusion in our Theorem \ref{thm:main}. On the other hand, \cite[Theorem 1.1 (i)]{CirMaggi} makes use only of stationarity.) The perturbation, given quantitatively in \cite{CirMaggi}, can be qualitatively described as follows: a large portion of $\p E$ is graphical over the spheres, the remaining part of $\p E$ is made of small connecting necks between distinct spheres. The Alexandrov deficit is scale invariant, so this description of $\p E$ is given in a normalised way, setting the spheres to have unit radius. In this normalised sense, the length of the connecting necks tends to $0$ as the Alexandrov deficit tends to $0$, and in the limit, as the deficit tends to $0$, one obtains a collection of spheres with equal radii touching tangentially.

We have the following corollary, which follows immediately from Theorem \ref{thm:main}, and from the results in \cite{CirMaggi} and \cite{BCW}.

\begin{cor}
\label{cor:droplets}
Let $E\subset \R^3$ be (the Lebesgue representative of) a Caccioppoli set, with $\overline{\p^* E} \subset B$ for some open ball $B$, $\overline{E}$ connected, and let $g:B \to \R$ be analytic. Assume that $E$ is stationary and stable with respect to $\mathcal{E}_B$ in the sense of Definition \ref{Dfi:station_stable_full}. There exists a constant $m_0>0$ such that if $|E|<m_0$ then $\p E = \p^* E$ is the graph of a function over a single sphere, with gradient bounded in the sup norm by $C \|g\|_{C^1} |E|^\beta$, for (dimensional) constants $C$ and $\beta$.
\end{cor}

\begin{proof}
Theorem \ref{thm:main} gives the fact that $\p^* E=\p E$ is smoothly embedded, which permits the use of \cite[Theorem 1.1 (i)]{CirMaggi} to conclude that $\p E$ is given by a perturbation of a finite collection of spheres, as described above. As $|E|$ tends to $0$ the Alexandrov deficit tends to $0$ (\cite[Section 1.2]{CirMaggi}). The Alexandrov deficit is scale-invariant and as $|E|\to 0$ the normalised spheres get closer and closer. This implies that, if a configuration with two or more spheres is possible, then the curvature of the normalised boundary, that we denote by $\p \hat{E}$, blows up as $|E|$ tends to $0$, while the mean curvature of $\p \hat{E}$ remains bounded. On the other hand, \cite[Theorem 1]{BCW} (see also \cite[Remark 3]{BCW}) provides a priori pointwise curvature bounds under a condition of stability for volume-preserving deformations, thus ruling out the possibility of two or more spheres. (The mean curvature bound required by \cite{BCW} is true by the normalization procedure that leads to $\hat{E}$; the mass bound required by \cite{BCW} is true in view of the fact that we have local convergence of $\p \hat{E}$ to two spherical caps with unit multiplicity, as $|E|\to 0$.) We thus find that under the assumptions of Corollary \ref{cor:droplets} there exists a sufficiently small volume regime in which there can only be a single sphere in the conclusion of \cite[Theorem 1.1 (i)]{CirMaggi}. We refer to \cite{CirMaggi} for the quantitative estimates on the function whose graph gives $\p E$. 
\end{proof}

\begin{oss}
If $E$ is assumed to have smooth embedded boundary and to be a local minimiser then the conclusion that there can be only a single sphere follows from \cite[Theorem 1.1 (ii)]{CirMaggi}, by means of density estimates for minimisers.
\end{oss}

An interesting question, not addressed here, is to quantify the smallness condition on the volume that ensures closeness to a single sphere in Corollary \ref{cor:droplets}.

\appendix
\section{Remarks on the assumptions of Theorem \ref{thm:main}}
\label{B}

\textit{Weaker variational assumptions.} It is possible to weaken the assumptions in Theorem \ref{thm:main}, as follows. In (i) of Definition \ref{Dfi:station_stable_full} it suffices to require (\ref{eq:stability}) for volume-preserving ambient deformations with initial velocity given by a $C^1_c$ vector field whose support is contained in an open set $W$ such that $\p^*E \cap W$ is a smoothly embedded surface. This set cannot be a priori characterised (hence the weaker assumption may be of interest only for more abstract purposes), however it follows from (\ref{eq:stationarity}), from Allard's theorem and from standard elliptic PDE theory that this set is not empty. The fact that this assumption suffices is due to the fact that it is also a sufficient assumption in Theorem \ref{1} (see the third hypothesis). In (ii) of Definition \ref{Dfi:station_stable_full}, when the starting configuration is \textbf{(a)} or \textbf{(b)} it suffices to require stationarity for coalescence and break up deformations. The stability condition in (ii) of Definition \ref{Dfi:station_stable_full} is only used to handle \textbf{(a')}.

\medskip

\textit{Analyticity of the potential.} 
The fact that $g$ is analytic is only used in deducing Theorem \ref{1} from the more general results in \cite{BW}, \cite{BW1}. Namely, to conclude that whenever two (non-identically coinciding) smooth embedded $2$-dimensional disks $D_1, D_2$ of mean curvature $g$ intersect only tangentially, then the coincidence set is an analytic set and therefore admits a stratification as the union of analytic submanifolds of dimensions $\in \{0, 1\}$. A $1$-dimensional coincidence set forces the existence of the local structure \textbf{(a)}, which is however ruled out by assumption in Theorem \ref{1}. It then follows that when $D_1$ and $D_2$ are separately smooth and stationary, $D_1 \cap D_2$ can only contain isolated points. This dimensional estimate allows the extension of the stability inequality, which is valid on the embedded parts of $D_1 \cup D_2$ thanks to the stability for ambient deformations, across those points (by a standard capacity argument). This extended inequality gives the hypothesis stated in \cite{BW}, \cite{BW1}, i.e.~a stability condition for the $C^2$-immersed portion of $\spt{V}$. 
The analysis of the coincidence set of two $C^2$ disks with mean curvature $\pm (\lambda-g)$ is the only step in which the analyticity of $g$ is used in the proof of Theorem \ref{thm:main}.
If $D_1$ and $D_2$ are only assumed to be $C^2$, then the coincidence set may be in principle more complicated.
Understading its structure is certainly an interesting problem, that would likely require PDE tools to be tackled.

\medskip

\textit{Dimensional restriction.}
We have worked in $\R^3$, which is the concretely relevant case to describe liquids. It is mathematically interesting to address the problem in arbitrary dimension. Many things remain unchanged, with the following exceptions. 

A higher-dimensional analogue of Theorem \ref{thm:main} would have to allow for a singular set of dimension $\leq n-7$ (see \cite{BW}, \cite{BW1}). This is well-known already in the minimising case. The higher-dimensional analogue of Theorem \ref{1} replaces \textbf{(a)} and \textbf{(b)} with the corresponding structures in which surfaces(-with-boundary) are replaced by hypersurfaces(-with-boundary) and $T$ has dimension $n-1$; this theorem then permits, in the conclusion, a singular set of dimension $\leq n-7$, where $\overline{\p^* E}$ is locally neither an embedding nor an immersion (at these points one has tangent cones but not tangent planes); moreover, there may be a set of immersed but non-embedded points of dimension $\leq n-2$, where $\overline{\p^* E}$ is locally given by two embedded disks that intersect tangentially. 
Then for a higher-dimensional analogue of Theorem \ref{thm:main} one would need to exclude more configurations (depending on the dimension of this intersection), rather than the single configuration \textbf{(a')}. This is likely possible although it probably requires a slight revision of the stability assumption in Definition \ref{Dfi:station_stable_one_sided}, by allowing the initial speed of the one-sided deformation to be unbounded (and belong to a suitable Sobolev space), or, by prescribing lower regularity in $(t,x)$ for the coalescence deformations. In $\R^3$ we required bounded initial speed for any virtual deformation allowed (since we wanted deformations that can be considered adimissible for a concrete liquid).

We also point out that a higher dimensional extension of Theorem \ref{thm:main} would also lead to a higher-dimensional version of Corollary \ref{cor:droplets}, since neither \cite{CirMaggi} nor \cite{BCW} require the dimensional restriction (in the case of \cite{BCW}, for $n\geq 7$ one would use \cite[Theorem 2]{BCW} rather than \cite[Theorem 1]{BCW}).

\section{Reminders on Caccioppoli sets}
\label{Caccioppoli}
A Caccioppoli set in $\R^{n+1}$, or set with locally finite perimeter, is a set $C \subset \R^{n+1}$ whose characteristic function $\chi_C$ is $BV_{\text{loc}}$, i.e.~the distributional gradient $D \chi_C$ is a Radon measure (regular and finite on compact sets). The perimeter of $C$ in a bounded open set $U$ is given by 
$$\text{Per}_U(C)=\int_U \|D \chi_C\| = \sup\left\{\int_C \text{div}\, T : T\in C^1_c(U;\R^{n+1}), \sup|T|\leq 1\right\}.$$
If $C$ has $C^2$ boundary this agrees with the usual perimeter $\mathcal{H}^{n}(\p C \cap U)$. This follows by observing firstly that, using $\sup|T|\leq 1$ and the divergence theorem, $|\int_C \text{div}\, T |\leq \mathcal{H}^{n}(\p C \cap U)$ for any allowed $T$; secondly, that by choosing $T$ to be a $C^1$ extension of the outer unit normal to $C$, suitably cut-off inside $U$, one gets as close as wished to the value $\mathcal{H}^{n}(\p C \cap U)$. If $\overline{C}\subset \subset U$ and $\p C$ is $C^2$, then a smooth estension of the unit outer normal used in place of $T$ gives exactly the full perimeter $\mathcal{H}^n(\p C)$.

The ``correct'' notion of boundary for a Caccioppoli set $C\subset \R^{n+1}$ is the so-called reduced boundary $\p^* C$, introduced by De Giorgi \cite{DG1}, \cite{DG2}. A Caccioppoli set is is defined up to $0$-measure sets, and the topological boundary changes with the representative. Moreover, it is seen by means of elementary examples that the topological boundary $\p C$ may be even the whole space. The reduced boundary $\p^* C$ is a measure-theoretic notion and it is invariant under the choice of representative. It is characterised by the existence of a well-defined measure-theoretic unit normal. De Giorgi's fundamental theorem guarantees the $n$-rectifiability of $\p ^* C$, namely the fact that $\p ^* C$ is, up to a set of vanishing $n$-dimensional measure, a countable union of Lipschitz images of Borel subsets of $\R^n$. The perimeter measure $-D \chi_C$ can be written as $({\Hc}^n \res \p^*C) \nu$, where $\nu$ is the measure-theoretic outer unit normal on $\p^* C$. If $C$ has $C^2$ boundary then $\p^* C = \p C$ and $\nu$ is the usual normal. 

\medskip

\medskip

\medskip

\noindent \textbf{Acknowledgments}. This work was partially supported by the EPSRC under grant EP/S005641/1. I would like to thank Giovanni Alberti, Guido De Philippis, Francesco Maggi, Robb McDonald and Helen Wilson for interesting and helpful conversations and comments on the topic. Further thanks to Timothy Johnston for useful discussions had while he worked on a summer project at UCL (EPSRC 2019 vacation bursary) under my supervision. This note was partially written while I was member of the Institute for Advanced Study, Princeton, in 2019: I gratefully acknowledge the excellent research environment and the support provided by the Institute and by the National Science Foundation under Grant No. DMS-1638352.


\begin{thebibliography}{99}


\bibitem{BarbDoCarmo} J. L. Barbosa, M. do Carmo {\it Stability of Hypersurfaces of Constant Mean Curvature} Math. Zeit. 185 (1984) 3 339-353. 
\bibitem{BW} C. Bellettini, N. Wickramasekera {\it Stable CMC integral varifolds of codimension $1$: regularity and compactness}, arXiv 2018.
\bibitem{BW1} C. Bellettini, N. Wickramasekera {\it Stable prescribed-mean-curvature integral varifolds of codimension $1$: regularity and compactness}, arXiv 2019.
\bibitem{BW2} C. Bellettini, N. Wickramasekera {\it The inhomogeneous Allen--Cahn equation and the existence of prescribed-mean-curvature hypersurfaces}, arXiv 2020.
\bibitem{BCW} C. Bellettini, O. Chodosh, N. Wickramasekera {\it Curvature estimates and sheeting theorems for weakly stable CMC hypersurfaces}, Adv. Math. 352 (2019) 133-157.
\bibitem{Brak} K. Brakke. {\it The motion of a surface by its mean curvature.} Princeton University Press (1978). 
\bibitem{CirMaggi} G. Ciraolo, F. Maggi {\it On the shape of compact hypersurfaces with almost constant mean curvature} Comm. Pure Appl. Math., 70 (2017), 665-716.
\bibitem{DG1} E. De Giorgi {\it Su una teoria generale della misura $(r-1)$-dimensionale in uno spazio ad $r$ dimensioni} Ann. Mat. Pura Appl. (4) 36 (1954),
191-213.
\bibitem{DG2} E. De Giorgi {\it Nuovi teoremi relativi alle misure $(r-1)$-dimensionali in uno
spazio ad $r$ dimensioni} Ricerche Mat. 4 (1955), 95-113.
\bibitem{DelgMag} M. Delgadino, F. Maggi {\it Alexandrov's theorem revisited}, Anal. PDE  12, No 6 (2019), 1613-1642.
\bibitem{Finn} R. Finn {\it Equilibrium capillary surfaces} vol. 284 of Grundlehren der Mathematischen Wissenschaften, Springer-Verlag, New York, 1986.
\bibitem{GonzMassTaman}  E. Gonzalez, U. Massari, I. Tamanini {\it On the regularity of boundaries of sets minimizing perimeter with a volume constraint} Indiana Univ. Math. J. 32 (1983), no. 1, 25-37. 
\bibitem{GonzMassTaman2} E. Gonzalez, U. Massari, I. Tamanini {\it Existence and regularity for the problem of a pendent liquid drop} Pacific J. Math. 88 (1980) n. 2, 399-420.
\bibitem{Maggi} F. Maggi {\it Sets of finite perimeter and geometric variational problems}, Cambridge Studies in Advanced Mathematics, {135}, {An introduction to geometric measure theory}, {Cambridge University Press, Cambridge}, {2012},  {xx+454}.
\bibitem{Mellet} A. Mellet {\it Some mathematical aspects of capillary surfaces} Singularities in mechanics: formation, propagation and microscopic description, 91-124, Panor. Synth\`{e}ses, 38, Soc. Math. France, Paris, 2012.
\bibitem{SimonNotes} L. Simon {\it Lectures on Geometric Measure Theory} Proceedings of the Centre for Mathematical Analysis 3, Canberra, (1984), VII+272.
%
\end{thebibliography}
\end{document}